\documentclass{article} 
\usepackage{amssymb}
\usepackage{amsmath}
\usepackage{latexsym}
\usepackage[mathscr]{euscript}
\usepackage{graphicx}
\usepackage{amscd}
\usepackage{amsthm} 
\usepackage{fullpage} 
\usepackage[all]{xy}
\usepackage{color}
\usepackage{wrapfig}
\newtheorem{theorem}{Theorem}
\newtheorem{corollary}[theorem]{Corollary}
\newtheorem{lemma}[theorem]{Lemma} 
\newtheorem{proposition}[theorem]{Proposition}

\newtheorem{definition}{Definition}

\newcommand{\bc}{\mathbb{C}}
\newcommand{\bp}{\mathbb{ P}}
\newcommand{\bz}{\mathbb{Z}}
\newcommand{\br}{\mathbb{R}}

\newcommand{\bs}{\mathbb{S}}

\newcommand{\cd}{\mathcal{D}}
\newcommand{\ce}{\mathcal{E}}

\newcommand{\cs}{\mathcal{S}}
\newcommand{\cg}{\mathcal{G}}

\newcommand{\cf}{\mathcal{F}}
\newcommand{\ct}{\mathcal{T}}
\newcommand{\bx}{\mathbb{X}}

\newcommand{\ca}{\mathcal{A}}

\newcommand{\hk}{\hookrightarrow}
\newcommand{\bg}{\bigskip}
\newcommand{\med}{\medskip}
\newcommand{\la}{\longrightarrow}
\newcommand{\bfl}{\begin{flushleft}}
\newcommand{\efl}{\end{flushleft}}

\newcommand{\eps}{\epsilon}

\newcommand{\xr}{\xrightarrow}

\newcommand{\bcp}{\bc \bp}

\newcommand{\ltm}{LM^{-TM}}

\newcommand{\sm}{\Sigma^\infty_M(LM_+)}
\newcommand{\smp}{\Sigma^\infty_M(P^{Ad}_+)}
\newcommand{\sxp}{\Sigma^\infty_X(P^{Ad}_+)}
\newcommand{\pad}{P^{Ad}}
\newcommand{\tv}{T_{vert}}

\newcommand{\mf}{\mathfrak}

\newcommand{\nee}{\nu_\epsilon (e)}
\newcommand{\mtm}{M^{-TM}}
\newcommand{\lom}{\Sigma^\infty (\Omega M_+)}
\newcommand{\omg}{\Sigma^\infty (G_+)}
\newcommand{\thomsn}{(\Omega S^n)^{\Omega \alpha}}
\newcommand{\thomk}{(\Omega S^{2k+1})^{\Omega (-\alpha_k)}}

\newcommand{\aop}{A \wedge A^{op}}

\newcommand{\bst}{\bs_{-\tau_M}}
\newcommand{\bse}{\bs_{\eta\circ f}}

\newcommand{\slm}{\Sigma^\infty (\Omega M_+)}
\newcommand{\negtm}{-\tau_M}
\newcommand{\sgm}{\Sigma^\infty(\Omega M_+)}

\usepackage{color}

\begin{document}

\title{ Twisted Calabi-Yau ring spectra, string topology,  and gauge symmetry} 
\author{Ralph L. Cohen \thanks{The first author was partially supported by a grant from the NSF.} \\ Department of Mathematics \\Stanford University \\ Bldg. 380 \\ Stanford, CA 94305   \and  Inbar Klang \thanks{The second author was supported by a Gabilan Stanford Graduate Fellowship.} \\ Department of Mathematics \\Stanford University \\ Bldg. 380 \\ Stanford, CA 94305}
\date{\today}
\maketitle 
\begin{abstract}  In this paper we import the theory of   ``Calabi-Yau" algebras and categories from symplectic topology and topological field theories, to the setting of spectra in stable homotopy theory.  Twistings in this theory will be particularly important.   There will be two types of Calabi-Yau structures in the setting of ring spectra: one that applies to compact algebras and one that applies to smooth algebras.  The main application of twisted compact Calabi-Yau  ring spectra  that we will study  is to   describe, prove,  and explain a certain duality phenomenon in string topology.  This is a duality between the manifold string topology of Chas and Sullivan \cite{chassullivan} and the Lie group string topology of Chataur-Menichi \cite{chataurmenichi}.  This will extend and generalize work of Gruher \cite{gruher}.      Then, generalizing work of the first author and Jones in \cite{cjgauge}, we show how the gauge group of the principal bundle  acts on this  compact Calabi-Yau structure, and we compute some explicit examples. 
We then extend the notion of the Calabi-Yau structure to smooth ring spectra,   and prove that Thom ring spectra of (virtual) bundles over the loop space, $\Omega M$,  have this structure.  In the case when $M$ is a sphere we will use these   twisted smooth Calabi-Yau ring spectra   to   study   Lagrangian immersions of the sphere  into its cotangent bundle.  We recast the work of Abouzaid and Kragh \cite{ak} to show that the topological Hochschild homology of the  Thom ring spectrum induced by the $h$-principle classifying map of the Lagrangian immersion, detects whether that immersion can be Lagrangian isotopic to an embedding. We then compute some examples.    Finally, we interpret  these Calabi-Yau structures directly in terms of   topological Hochschild homology and cohomology. 
    \end{abstract}

\tableofcontents

\section*{Introduction} 
The theory of Calabi-Yau algebras and categories have proven to be very important in symplectic topology and the study of topological field theories
\cite{costello}, \cite{ks}, \cite{kv}, \cite{lurie}, \cite{cg}.  One of the goals of this paper is to adapt this theory to the setting of spectra in stable homotopy theory, and to apply it to  prove and explain a 
  duality relationship between the string topology of a manifold, and the string topology of a classifying space of a compact Lie group.  We also use this notion to study Lagrangian immersions of spheres.

 By way of background, recall that ``string topology" is a term that was originally coined by Chas and Sullivan in their seminal paper \cite{chassullivan}.  In that paper this term referred to certain algebraic properties of the homology of the loop space of a closed, oriented manifold, $H_*(LM)$,  that were the  result of a type of intersection theory in $LM$.  This intersection theory came about by studying the fibration $\Omega M \to LM \xr{ev} M$, where $ev$ evaluates a loop
at $1 \in S^1$.  Even though the loop space  is itself  infinite dimensional,  the intersection theory defining the string topology operations is ultimately possible because of the finite dimensionality and compactness of $M$, as well as the  fiberwise multiplicative properties of this fibration. 

Since that time the subject has expanded considerably. An important variation of the string topology intersection theory was described by Chataur and Menichi in \cite{chataurmenichi} where they defined operations on the cohomology of the loop space of the classifying space of a compact Lie group, $LBG$.  In this setting the analogous fibration, $G \to LBG \xr{ev}  BG$
is studied, and the intersection theory defining these operations is possible because of the compactness of $G$ as well as the fiberwise multiplicative properties of this fibration.  A theory that includes both the string topology of a manifold and that of classifying spaces was developed in the setting of stacks in \cite{LUX} and \cite{BGNX}.  In this setting the   intersection theory is done in an appropriate algebraic geometric category.

An observation that helped to shed light on this intersection theory was made by the first author and Klein in \cite{CohenKlein} when they classified ``umkehr maps" that satisfy appropriate naturality and linearity properties.  This led to the observation that the ring spectrum $LM^{-TM}$,  which was shown to realize the Chas-Sullivan loop product by the first author and Jones in \cite{cohenjones}, can be viewed as a twisted generalized cohomology theory evaluated on  the manifold $M$.
Specifically, if one takes the fiberwise suspension spectrum of the fibration $\Omega M \to LM \xr{ev} M$, and denotes the resulting parameterized spectrum by the notation $\Sigma^\infty(\Omega M_+) \to \Sigma^\infty_M(LM_+) \to M,$ then the result  is a parameterized ring spectrum which defines  a twisted cohomology theory $\cs^\bullet_M$ from the category of  spaces over $M$, $\ct_M$,  to  an appropriate category of spectra.   
  If $f : X \to M$ is an object in $\ct_M$, then $\cs_M^\bullet (X,f)  = \Gamma_X(f^*(\sm))$, the spectrum of sections over $X$ of the pullback via $f$  of the parameterized spectrum $\sm \to M$.  See \cite{CohenKlein} and \cite{MS} for details.   Since $\sm$ is a parameterized ring spectrum, this spectrum of sections  inherits a ring spectrum structure.   Moreover it was proved in \cite{CohenKlein} the value of this cohomology on the identity map $id : M \xr{=} M \in \ct_M$, has the homotopy type
$$
\cs_M^\bullet (M)  = \Gamma_M(\sm) \simeq LM^{-TM}
$$
as ring spectra.  This equivalence is a type of twisted Poincar\'e or Atiyah duality as explained in \cite{CohenKlein}. Moreover one sees that the string topology intersection pairing (loop product) on $H_*(LM^{-TM}) \cong H_{*+n}(LM)$ corresponds, via this twisted Poincar\'e duality, to a generalized cup product pairing in the cohomology $\cs_M^\bullet (M)$. This is a twisted generalization of the well-known phenomenon that the intersection product in $H_*(M)$ corresponds up to sign  under traditional Poincar\'e duality, to the cup product in $ H^*(M)$. 

As observed by Gruher and Salvatore \cite{gruhersalvatore}, the string topology product exists in the presence of any
fiberwise monoid over a closed manifold,  $Q \to E \to M$.  Here $Q$ is a monoid, and the bundle $E$ comes equipped  with a fiberwise product $E \times_M E \to E$ over $M$, consistent with the monoid structure of the fiber $Q$.  In this case the Thom spectrum $E^{-TM}$ is a ring spectrum. It was also observed in \cite{gruhersalvatore} that principal bundles $G \to P \to M$ give rise to fiberwise monoids by taking
the associated adjoint bundle, $G \to P^{Ad} \to M$ where $P^{Ad} = P \times_G G^{Ad}$.  Here $G^{Ad}$ denotes $G$ with the left $G$-action given by conjugation.  

 As observed in \cite{cjgauge}, the string topology of principal bundles over manifolds can also be represented by twisted cohomology theories.   The representing parameterized spectrum is the fiberwise suspension spectrum
$\Sigma^\infty (G_+) \to \Sigma^\infty_M (P^{Ad}_+) \to M$.  Let $\cs_P^\bullet$ denote the corresponding twisted cohomology theory.  In particular
$$
\cs_P^\bullet (M) = \Gamma_M(\smp) \simeq (P^{Ad})^{-TM},
$$
and the ring structure comes from a generalized cup product on 
$\smp^\bullet (M)$.  We refer to $\cs_P^\bullet (-)$   as the ``{\it manifold string topology}" structure on the principal bundle $P$. 

 This perspective on the string topology spectrum,  $LM^{-TM}$, or more generally  $ (P^{Ad})^{-TM}$, in terms of the sections   of a parameterized   spectrum was particularly useful in \cite{cjgauge}, where the units of these ring spectra were studied.  In particular
it was shown that the gauge group $\cg (P)$ of the principal bundle acts naturally on the string topology spectrum, and so there is a homomorphism,
$$
\cg (P) \to GL_1((P^{Ad})^{-TM})
$$
which was studied and computed in \cite{cjgauge}.

\med
The first goal of this paper is to show that there is a dual construction for the string topology of the classifying space of a compact Lie group, to  investigate this duality using a stable homotopy theoretic version of compact Calabi-Yau algebras,   and to compute some of its properties  including  gauge symmetry.

\med
We now state the results more precisely.  Let $G$ be a compact Lie group, and let $G \to P \to X$ be a principal $G$-bundle.  In this, $X$ can be any space of the homotopy type of a $CW$-complex.  It need not be finite.  In particular, an important example is the universal principal bundle $G \to EG \to BG$.   As before,
let $G \to P^{Ad} \to X$ be the corresponding adjoint bundle.  Recall that in the case of the universal  bundle, $EG^{Ad} \simeq LBG$. 

 Consider the fiberwise suspension spectrum,
$$
\Sigma^\infty (G_+) \to \smp \to X,
$$  and   let $\cd (\smp)$ be the fiberwise Spanier-Whitehead dual, as in \cite{MS}.  This is a parameterized spectrum over $X$, whose fibers are
the Spanier-Whitehead duals of the fibers of $\smp$:
$$
G^\vee \to \cd (\smp) \to X,
$$ where  $G^\vee = Map (\Sigma^\infty (G_+), \bs)$.  Here $\bs$ denotes the sphere spectrum.   Notice that $G^\vee$ is a coalgebra spectrum, with coalgebra structure dual to the ring structure on $\Sigma^\infty (G_+)$. 

We denote the twisted homology theory associated to this parameterized spectrum by $\cs^P_\bullet : \ct_X \to Spectra$.  The following will be proved in section 1.

\med

\begin{theorem}\label{main}  The parameterized spectrum $\cd (\Sigma^\infty (G_+)) \to \cd(\smp) \to X$ is a weak fiberwise coalgebra spectrum satisfying the following properties.
\begin{enumerate}
\item Let $f : Y\to X$ be an object in $\ct_X$.  Then the induced twisted homology
$ 
\cs^P_\bullet (Y, f)   
$ 
is a weak coalgebra spectrum.    
\item There is an equivalence of spectra,
$$
\alpha : (P^{Ad})^{-\tv}   \xr{\simeq}  \cs^P_\bullet (X)  
$$
where
$ (P^{Ad})^{-\tv}$ is the Thom spectrum of  minus the vertical tangent bundle $\tv \pad \to \pad$. Furthermore a Pontrjagin-Thom construction gives $(P^{Ad})^{-\tv}$ a natural coproduct which is taken by $\alpha$ to the coproduct in $\cs^P_\bullet (X) $.
\item  If one takes the cohomology of the coalgebra spectrum, $H^*(\cs^P_\bullet (Y, f) ; k )$ (here the coefficients are in a field $k$), one obtains a   graded algebra, $$
H^*(\cs^P_\bullet (Y) ; k) \otimes H^*(\cs^P_\bullet (Y); k )   \to H^*(\cs^P_\bullet (Y) ; k)
$$
which we call the ``Lie group string topology algebra of $f^*(P)$".    Using the equivalence in part 2, then when the vertical tangent bundle $T_{vert} \to P^{Ad}$
is orientable, one obtains a graded algebra  of degree $-d$, where $d = dim \, G$:
$$
H^p(\pad) \otimes H^q(\pad) \to H^{p+q-d}(\pad).
$$

\item In the case of the universal principal bundle $G \to EG \to BG$,  this algebra is isomorphic to the algebra structure in the string topology of the classifying space $BG$ as described by Chataur and Menichi \cite{chataurmenichi}
$$
H^*(\cs^{EG}_\bullet (BG)) \cong H^*(LBG).
$$
  \end{enumerate}
\end{theorem}

\bg
\noindent {\bf Comments:}  
\begin{enumerate}
\item The notion of a ``weak" fiberwise coalgebra spectrum will be defined in section 1.
\item We refer to the coalgebra spectrum     $\cs^P_\bullet (X)   \simeq  (P^{Ad})^{-\tv}$ as the  ``{\it Lie group string topology spectrum of the principal bundle $P$}".
\item The equivalence $\alpha : (P^{Ad})^{-\tv}   \xr{\simeq}   \cs^P_\bullet (X) = \cd (\smp)/X   $ can be viewed as a fiberwise Atiyah duality, which on the level of fibers
is the classical Atiyah equivalence \cite{atiyahdual}, $$\alpha :  G^{-TG} \simeq  \Sigma^{-\mf{g}}(G_+)   \xr{\simeq}      G^\vee $$  where $G^{-TG}$ is the Thom spectrum of minus the tangent bundle, which is equivariantly equivalent to the  desuspension of $\Sigma^\infty(G_+)$ by the adjoint representation of $G$ on the Lie algebra  $\mf{g}$.  
\item  The fact that the  cohomology  algebra $H^*(LBG^{-\tv}) \cong H^{*+d}(LBG)$ is the string topology of classifying spaces was  proved by Gruher in \cite{gruher}.    
\end{enumerate}

\bg 
Once this theorem is established we restrict to the situation where we have a principal $G$ - bundle over a closed manifold: $G \to P \to M$.  In this case we can study   both the ``manifold  string topology structure" of $P$, that is, the twisted cohomology theory $$\cs_P^\bullet (M) = \Gamma_M(\smp) \simeq (P^{Ad})^{-TM}$$ as well as the ``Lie group string topology structure" of $P$, which is to say the twisted homology theory $$ \cs^P_\bullet (M)   \xr{\simeq}  (P^{Ad})^{-\tv}.$$   The following is a consequence of Theorem \ref{main} as well as Gruher's work \cite{gruher}. 

\med
\begin{corollary}\label{frobenius} Let $G \to P \to M$ be a principal bundle where $G$ is a compact Lie group of dimension $d$ and $M$ is a closed manifold of dimension $n$.   The string topology spectra  $\cs_P^\bullet (M) \simeq (P^{Ad})^{-TM}$ and    $ \cs^P_\bullet (M) \simeq (\pad)^{-\tv},$    are Spanier--Whitehead dual to each other with the algebra structure of the former corresponding to the coalgebra structure of the latter under this duality.     When one applies homology, this gives $H_*(\pad)$ the structure of a Frobenius algebra of dimension $n-d$.   The multiplication in this Frobenius algebra comes from the  manifold string topology, and the comultiplication comes from the  Lie group string topology.      
\end{corollary}

\med
 In  section 2  we will  define the notion of a  ``twisted compact Calabi-Yau" ring spectrum (``twisted cCY"), which can be viewed as a strengthened, derived version of Frobenius algebra in the category of spectra. This definition  is adapted from the notion  of a ``compact Calabi-Yau algebra" defined by Kontsevich and Soibelman \cite{ks}, as a way of studying two dimensional topological field theories.  (We note that Kontsevich and Soibelman used different terminology for this concept.) Related notions were defined by Costello \cite{costello} and Lurie \cite{lurie}. In these definitions, the algebra (or ring spectrum) involved must satisfy a finiteness condition called ``compactness". In the spectrum setting this means that the spectrum is a perfect module over the sphere spectrum.  In our definition of this structure  in the setting of spectra, a key role is played by a ``twisting bimodule" over the compact ring spectrum. The following is the main result of this section.

 \med
 \begin{theorem}\label{Frob}  Let $G \to P \to M$  be a principal bundle with compact Lie group fiber and closed manifold base.  Then the manifold string topology
 $\cs_P^\bullet (M)$ naturally admits the structure of  a twisted, compact Calabi-Yau ring spectrum of dimension $n-d$.   The Lie group string topology spectrum $ \cs^P_\bullet (M)$  is the twisting bimodule spectrum in this structure.  Moreover if $E_*$ is a generalized homology theory with respect to which both the vertical tangent bundle   $T_{vert} \to \pad$ and the tangent bundle $TM \to M$ are oriented, then the Calabi-Yau structure on  $\cs_P^\bullet (M)$ induces a Frobenius algebra structure on the homology  of the manifold string topology, $E_*(\cs_P^\bullet (M))$ whose dual is the homology of the Lie group string topology spectrum, $E_*(\cs^P_\bullet (M))$.   
  \end{theorem}
  
\med

In \cite{cjgauge} an action of the gauge group $\cg (P)$ of the principal bundle $G \to P \to M$  on the manifold string topology spectrum $\cs_P^\bullet (M) = (\pad)^{-TM}$ was described and computed.    
   In section 3   we  use Theorems \ref{main} and \ref{Frob}  to describe a similar action of $\cg (P)$ on the Lie group string topology spectrum $\cs^P_\bullet (M) = (\pad)^{-\tv}$.    We also show that this gauge symmetry respects the Calabi-Yau structure.  See Theorem \ref{gaugesym} below for a precise statement.   We then compute some   explicit examples  of this gauge symmetry.

In section 4 we introduce the related notion of twisted \sl smooth \rm Calabi-Yau  ring spectra.   Smoothness is  a different form of smallness property than compactness.  A ring spectrum $A$ is \sl smooth \rm if it is perfect as a bimodule over itself.  That is, it is a perfect as a left $(A \wedge A^{op})$-module spectrum.   The spectrum notion of a ``twisted sCY"    structure, is adapted   from the notion of ``sCY" algebras and categories, that was first defined by Kontsevich and Vlassopolous \cite{kv}  and used by the first author and Ganatra in \cite{cg} to compare the string topology topological field theory to the Floer symplectic field theory of  cotangent bundles.   In the spectral theory  a twisting bimodule spectrum plays an important role.  We show that this structure  occurs in certain Thom spectra of virtual bundles over the based loop space of a manifold, $\Omega M$.  That is we prove the following  theorem:

\begin{theorem}\label{smooth}  Let  $M$ be a closed manifold, and $f : M \to BBO$ be a map to a delooping of $BO$.  Here, by Bott periodicity we may take $BBO$ to be   the infinite homogeneous space $SU/SO$.  Consider the induced map of loop spaces, $ \Omega f : \Omega M \to BO$. Then its Thom spectrum, which we denote by $(\Omega M)^{\Omega f}$,  naturally admits the structure of a twisted, smooth Calabi-Yau ring spectrum.
\end{theorem}

\med
\noindent \bf Remark. \rm  
When $f : M \to BBO$ is the constant map, this theorem implies that the suspension spectrum $\Sigma^\infty (\Omega M_+)$ has the structure of a twisted sCY ring spectrum.  This strengthens a result of the first author and Ganatra in \cite{cg} saying that the singular chain complex $C_*(\Omega M)$ admits the structure of a smooth Calabi-Yau differential graded algebra.

\med
Also in section 4, we describe how these ring spectra arise naturally in the study of Lagrangian immersions.  In particular, for the case of spheres, we combine the results of Abouzaid and Kragh \cite{ak} with those of the first author, Blumberg, and Schlichtkrull \cite{BCS} to prove the following (see Theorem \ref{lagTHH} for a more precise statement).

\med
\begin{theorem} Associated to a  Lagrangian immersion $\phi  : S^n \to T^*S^n$  there is a loop map $\Omega \alpha_\phi : \Omega S^n \to BU$.  If the Lagrangian immersion $\phi$ is Lagrangian isotopic to a Lagrangian embedding,  then there is an equivalence of topological Hochschild homology spectra,
$$
THH ((\Omega S^n)^{\Omega \alpha_\phi}) \simeq THH(\Sigma^\infty (\Omega S^n_+)).
$$
\end{theorem}

\med
We then use this theorem, together with homotopy theoretic results about the image of the $J$-homomorphism to recast results in \cite{ak} giving examples of Lagrangian immersions of spheres that are not Lagrangian isotopic to embeddings, but \sl are \rm smoothly isotopic to embeddings.


\medskip

\med
Finally in section 5 we describe this structure from the perspective of  topological Hochschild (co)homology.  More specifically, let   $G \to P \to M$ be a smooth principal bundle, where $G$ is a compact Lie group and $M$ is a smooth, closed manifold.  Let  $h : \Omega M \to G$ be the holonomy of a connection on $P$.  This induces a map of ring spectra   $h : \Sigma^\infty (\Omega M_+) \to \Sigma^\infty (G_+)$.    Thus $h$ defines bimodule structures on  $\Sigma^\infty (G_+)$ over $\Sigma^\infty (\Omega M_+)$. The main result of this section is  the following.

\med

\begin{theorem}\label{hochschild}  We have the following equivalences involving topological Hochschild homology $THH_\bullet$ and topological Hochschild cohomology $THH^\bullet$.
\begin{enumerate} 
\item $THH_\bullet(\Sigma^\infty (\Omega M_+), \Sigma^\infty (G_+)) \simeq \Sigma^\infty(\pad_+)$ \notag \\
\item $ THH^\bullet (\Sigma^\infty (\Omega M_+), \Sigma^\infty (G_+)) \simeq (\pad)^{-TM}  \simeq \cs_P^\bullet (M).   $ \quad \text{This equivalence is one of ring spectra.} \\
\item $THH_\bullet(\Sigma^\infty (\Omega M_+), G^\vee) \simeq (\pad)^{-\tv} \simeq \cs_\bullet^P(M).$  \quad \text{This equivalence is one of coalgebra spectra.} 
\end{enumerate}
\end{theorem}

\med
We end by describing the twisted Calabi-Yau structure on the string topology spectrum from the perspective of these topological Hochschild homology spectra.  A consequence of the resulting duality properties is the following:

\med
\begin{corollary}\label{hochschild2}
If $M$ is oriented, there is a  nondegenerate bilinear form on  Hochschild homology, 
$$HH_*(C_*(\Omega M), C_*(G))  \times HH_*(C_*(\Omega M), C_*(G)) \to  k. $$ That is, this Hochschild homology space is self dual.
\end{corollary}

\paragraph{Acknowledgments.} The second author would like to thank Arpon Raksit for illuminating conversations.

\section{A twisted homology theory representing Lie group  string topology  }   The goal of this section is to describe Lie group string topology as a twisted generalized homology theory, and to  prove Theorem \ref{main}.   The main issue in proving this theorem is to describe a parameterized form of Atiyah duality.    We begin by recalling the  specific map yielding the Atiyah duality between the Thom spectrum of minus the tangent bundle of a closed manifold $M$, and the Spanier-Whitehead dual of $M$ \cite{atiyahdual}, \cite{cohenatiyah}.

\med
Let $M^n$ be a closed $n$-dimensional manifold and  $e : M \hk \br^k$ be an embedding into Euclidean space with normal bundle  $\eta_e \to M$.   
By the tubular neighborhood theorem, for sufficiently small $\eps > 0$, the open set $\nee \subset \br^k$ consisting of points within a distance of $\eps$ of $e(M)$ can be identified with the total space $\eta_e$.  

Consider the map
\begin{align}\label{alex}
\alpha : \left( \br^k - \nee\right) \times M & \to \br^k - B_\eps (0) \simeq S^{k-1} \notag \\
(v, y)  &\la v - e(y)
\end{align} where $B_\eps (0)$ is the open ball of radius $\eps$.
This map induces the Alexander duality isomorphism
$$
\begin{CD}
 \tilde H_q(\br^k - e(M)) \cong \tilde H_q(\br^k - \nee) @>\cong >> \tilde  H^{k-q-1}(M).
\end{CD}
$$

Atiyah duality \cite{atiyahdual}  is induced by   the same map:
\begin{align}
M^{\eta_e} \wedge M_+ \cong (\br^k \times M )/\left( (\br^k - \nee) \times M \right)  &\la \br^k /( \br^k - B_\eps (0)) \quad \cong S^k  \notag \\
(v, y)   &\la v-e(y).
\end{align}
The adjoint of this map gives a map from the Thom space of $\eta_e$ to the mapping space,
$ 
\alpha : M^{\eta_e}  \la Map(M, S^k)
$
which defines  the Atiyah duality  equivalence of spectra,
\begin{equation}\label{atdual}
\alpha : \mtm \la Map(M, \bs).
\end{equation}  Here this notation refers to the mapping spectrum between the suspension spectrum $\Sigma^\infty (M_+)$ to the sphere spectrum $\bs$.  This is the Spanier-Whitehead dual of $M$, and will  be denoted by $M^\vee$.  
Indeed in \cite{cohenatiyah} the author constructed a symmetric ring spectrum (without unit), $\mtm$.  The  $k^{th}$ space of this spectrum is equivalent, through a range of dimensions that increases with $k$,  to the Thom space $M^{\eta_e}$ and is constructed by allowing the embeddings and the choices of $\eps$ to vary.  The $k^{th}$ space of the mapping spectrum $Map(M,\bs)$ has the homotopy type of $Map(M, S^k)$.  It was shown in \cite{cohenatiyah}
that the map $\alpha$ induces an equivalence of symmetric ring spectra. We refer the reader to \cite{cohenatiyah} for details.    

We now pass to the parameterized setting.  Our goal is to describe a parameterized form of this Atiyah duality equivalence. Let $G \to P \to X$ be a principal bundle with compact Lie group fiber. By the fiberwise duality theorem of May and Sigurdsson (Theorem 15.1.1 of \cite{MS}), the parameterized suspension spectrum $\Sigma^\infty (G_+) \to \Sigma^\infty_X(P^{Ad}_+) \to X$ is (fiberwise) dualizable because each fiber spectrum is dualizable.  This in turn is because every fiber spectrum is equivalent to $\Sigma^\infty (G_+) $,  which is dualizable since $G$ is compact.  The parameterized Spanier-Whitehead dual is what we called $G^\vee \to \cd (\sxp) \to X$ in the introduction. The construction in \cite{MS} is quite general, in this particular case, however,  we will  describe this fiberwise dual   explicitly.   

The spectra we work with will be orthogonal spectra, and when we describe a group action, we use $RO(G)$-indexed orthogonal spectra.  We refer the reader to \cite{MM} for details.

Recall that $\pad = P \times_G G^{Ad}$.  Let $V$ be a finite dimensional orthogonal representation of $G$, and let $S^V = V \cup \infty$ be the one-point compactification where the $G$-action fixes $\infty$. The conjugation action of $G$ on itself defines an action of $G$ on $Map (G, S^V)$,    
\begin{align}\label{gact}
g \cdot \phi : G &\to S^V\\
h &\to g\phi (g^{-1}hg).  \notag 
\end{align}
This defines an $RO(G)$-graded $G$-spectrum, which we call $G^\vee$.

We defined the  parameterized spectrum $\cd (\sxp)$ as an $RO(G)$-graded spectrum.  For a representation $W$, the $W$-space is defined to be 
\begin{equation}\label{ksp}
\cd (\sxp)_W = P \times_G Map(G, S^W)
\end{equation}
which fibers over $X = P/G$ with fiber $ Map(G, S^W) = Map (G, S^k)$, where $k =  dim \, W$.    The fiberwise suspension by a representation $U$ is given by  $\Sigma^U_X( \cd (\sxp)_W )$ is $P\times_G  S^U\wedge(Map(G, S^W))$ and the structure map $ \eps_U : \Sigma^U_X( \cd (\sxp)_W ) \to  \cd (\sxp)_{W\oplus U} $
is induced by the $G$-equivariant map
\begin{align}
\eps_U : S^U \wedge (Map (G, S^W))   &\to Map (G, S^{W\oplus U})\notag \\
\eps_U (t\wedge \phi)(g) &= \phi (g)  \wedge t . \notag
\end{align}  

\med
Notice that since the multiplication map $G \times G \to G$ is equivariant with respect to the adjoint action (the action on $G\times G$ is diagonal),   the induced comultiplication in the Spanier-Whitehead dual spectrum $G^\vee \to G^\vee \wedge G^\vee$ is also equivariant, and so induces a 
weak 
fiberwise coalgebra structure on the parameterized spectrum
$\cd (\sxp)$.  

By a  \sl ``weak fiberwise coalgebra" \rm structure on a parameterized spectrum, we simply mean the following.

\begin{definition}\label{coalgebra} A parameterized spectrum $ E \to \ce \to X$ is a ``weak fiberwise coalgebra" if there is a ``comultiplication" map
$\gamma : \ce \to \ce \wedge_X \ce$  and a ``counit"  $\eta : \ce \to \bs_X$ in the category of parameterized spectra over $X$, that satisfy the usual
co-associativity and co-unit properties up to homotopy.  No coherence conditions on the homotopies are assumed.   Here $\bs \to \bs_X \to X$ is the parameterized sphere spectrum.  Namely the $n^{th}$ space of $\bs_X$ is  $X \times S^n$.
\end{definition}

Notice that given a 
fiberwise coalgebra spectrum $E \to \ce \to X$, then for any object $f : Y\to X$ in $\ct_X$, the twisted homology spectrum
$\ce_\bullet (Y, f) = \ce/X$ is an ordinary 
 coalgebra spectrum.  

\med
The source of the parameterized Atiyah duality map is a parameterized Thom spectrum.  More precisely, let $e : G \subset V$ be an equivariant
embedding of $G$ with its conjugation action into a finite dimensional $G$-representation $V$.  Let $k = dim \, V$.  
  Let $\nee $ be an equivariant tubular neighborhood as above. It is equivariantly diffeomorphic  to the normal bundle $\eta_V\to G$. (We are suppressing the embedding $e$ from the notation.)
We let $\eta_V^{vert} \to \pad$ be the vector bundle
\begin{equation}
\eta_V^{vert} = P \times_G  \eta_V \to P\times_G G^{Ad} = \pad 
\end{equation}
 The fiberwise Thom space of this bundle is 
  homeomorphic to  the fiberwise one-point compactification of the tubular neighborhood,
  $$
  P \times_G  G^{\eta_V} \cong  P \times_G (\nee \cup \infty).
  $$
  Notice also that there is a map from the fiberwise suspension 
  $$
  \eps_W : \Sigma^W_X (P \times_G G^{\eta_V}) = P \times_G  \left(S^W \wedge G^{\eta_V}\right) \xr{\simeq}  P \times_G G^{\eta_{V\oplus W}}
  $$
   This data defines an $RO(G)$-graded parameterized spectrum    $   P \times_G G^{-TG}$ over $X$ whose $W^{th}$ space is $\Omega^{V\oplus W} (P \times_G G^{\eta_{V\oplus W}})$. Here, for a representation $U$,  $\Omega^U$ refers to the $U$-fold loop space, $Map_\bullet (S^U,  - )$.

   Furthermore, the Atiyah duality map described above defines a map $\alpha : G^{\eta_V} \cong  \nee \cup \infty \to Map (G,  S^V)$.      This map   is equivariant, and so defines  Atiyah duality maps $\bar{\alpha_V} : P \times_G G^{\eta_V} \to  P \times_G Map (G, S^V)$.  These maps respect the spectrum structure maps and so prove the following:

\med

\begin{lemma}\label{paramAtiya}

 The maps $\bar{\alpha_k}$ define an equivalence of parameterized spectra over $X$,
   $$
   \bar \alpha : P \times_G G^{-TG} \xr{\simeq}   \cd (\sxp).
   $$
\end{lemma}

\med
   The map $\bar \alpha$ therefore defines an equivalence of the generalized twisted homology theories these parameterized spectra represent.  Given an object $f : Y \to X$ in $\ct_X$,  the twisted homology the parameterized spectrum $\cd (\sxp)$ represents is what we called $\cs_\bullet^P(Y, f)$ in the introduction. The twisted homology the parameterized spectrum $P \times_G G^{-TG}$ represents is given by the (ordinary) spectrum $f^*(P)_+ \wedge_G G^{-TG}$, which is the Thom spectrum of the virtual bundle $-f^*(T_{vert})$ over $f^*(P^{Ad})$, where $T_{vert} \to \pad$ is the vertical tangent bundle.  Applying this to $X$ itself, we have the equivalence of spectra,
\begin{equation}\label{-tvert}
\bar{\alpha} : (\pad)^{-T_{vert}}   \xr{\simeq}\cs^P_{\bullet}(X)
\end{equation}

\med
Now recall that given any map $\phi : M \to N$ between closed manifolds, the Pontrjagin-Thom construction defines a map $\tau_g: N^{-TN} \to M^{-TM}$ making the following diagram of spectra homotopy commute:
\begin{equation}\label{atiyahcompat}
\begin{CD}
N^{-TN} @>\tau_g >> M^{-TM}  \\
@V\alpha V\simeq V     @V\simeq V\alpha V \\
N^\vee   @>> g^\vee >  M^\vee
\end{CD}
\end{equation}
 
Applying this to the multiplication map $\mu : G \times G \to G$, we get a homotopy commutative diagram
$$
 \begin{CD}
G^{-TG} @>\tau_\mu >> G^{-TG}\wedge G^{-TG}  \\
@V\alpha V\simeq V     @V\simeq V\alpha V \\
G^\vee   @>> \mu^\vee >  G^\vee \wedge G^\vee
\end{CD}
$$

Given the adjoint action of $G$ on itself, and the diagonal adjoint action of $G$ on $G \times G$,  the multiplication map $\mu : G \times G \to G $ is equivariant. Therefore  there is an induced   fiberwise coproduct  on the parameterized spectrum $P \times_G G^{-TG}$, as there is on 
$ \cd (\sxp)$.  

We now verify that the induced map $  \bar \alpha : P \times_G G^{-TG} \xr{\simeq}   \cd (\sxp)$  preserves these coproducts. 
We do this by studying the definition of the maps involved more carefully.  
  
Toward this end  let $e : G \subset V$ be an equivariant
embedding of $G$ with its conjugation action into a finite dimensional $G$-representation $V$, as above.    Let $k = dim \, V$. We then have an induced composition of equivariant embeddings, 
\begin{equation}
G \times G    \xr{\mu \times e \times e} G \times V \times V \xr{e \times 1 \times 1} V \times V \times V.
\end{equation}

Recall that the tangent bundle of $G$ has an equivariant trivialization $TG \cong G \times \frak{g}$
where $\frak{g}$ is the  Lie algebra with its adjoint action.  Differentiating $e : G \hk V$ at the identity gives a linear equivariant embedding
$\frak{g} \hk V$.  We let $\frak{g}^\perp$ be the orthogonal complement with its induced action.  

The total space of the normal bundle of $G \times V \times V \xr{e \times 1 \times 1} V \times V \times V$ is clearly equivariantly isomorphic to 
$G \times \frak{g}^\perp \times V \times V$.  We perform the Pontrjagin-Thom construction on the induced (equivariant) embedding
of the restriction of the total space 
$$
(\mu \times e \times e)^* (G \times \frak{g}^\perp \times V \times V)    \hk G \times \frak{g}^\perp \times V \times V.
$$
This is a codimension $2k-d$ embedding.  The Pontrjagin-Thom construction gives an equivariant map
$$
\tau_\mu^{V\times V\times V} : G_+ \wedge S^{\frak{g}^\perp} \wedge  S^V \wedge S^V      \la (G\times G)_+  \wedge S^{\frak{g}^\perp} \wedge S^{\frak{g}^\perp} \wedge S^V,
$$ or equivalently, 
$$
\tau_\mu^{V\times V\times V}  : G^{\eta_e} \wedge S^V \wedge S^V      \la   G^{\eta_e} \wedge G^{\eta_e} \wedge S^V.
$$
This defines the map
$$
\tau_\mu^{V\times V\times V}  :  P \times_G \left(G^{\eta_e} \wedge S^V \wedge S^V \right)     \la P \times_G \left( G^{\eta_e} \wedge G^{\eta_e} \wedge S^V\right).
$$ 

Similarly, the Atiyah duality map, which as discussed above is defined via a Pontrjagin-Thom collapse, is an equivariant  map
$$
\alpha : G^{\eta_e} \wedge S^V \wedge S^V   \to Map (G, S^V) \wedge S^V \wedge S^V,  
$$
which induces a map
$$
\alpha : P \times_G \left(G^{\eta_e} \wedge S^V \wedge S^V \right) \la P \times_G \left(Map (G, S^V) \wedge S^V \wedge S^V\right).
$$
The compatibility of these Pontrjagin-Thom maps  yields that the following diagram  commutes:

\begin{equation} 
\begin{CD}
  P \times_G \left(G^{\eta_e} \wedge S^V \wedge S^V \right)    @>\tau_\mu^{V\times V\times V}>>  P \times_G \left( G^{\eta_e} \wedge G^{\eta_e} \wedge S^V\right) \\
  @V\alpha VV     @VV\alpha V  \\
P \times_G \left(Map (G, S^V) \wedge S^V \wedge S^V\right)    @>>\mu^\vee >   P \times_G \, \left(Map (G \times G,  S^V \wedge S^V) \wedge S^V\right).  
\end{CD}
\end{equation}

Passing to spectra,  this says that the following diagram of parameterized spectra over $X$ homotopy  commutes: 
 \begin{equation}
 \begin{CD}
 P \times_G G^{-TG} @>\tau_\mu >> P \times_G G^{-TG}\wedge G^{-TG}  \\
@V\alpha V\simeq V     @V\simeq V\alpha V \\
P\times_G G^\vee   @>> \mu^\vee >  P\times_G (G^\vee \wedge G^\vee)
\end{CD}
  \end{equation}  Or, written with the notation used above, the following diagram of parameterized spectra over $X$ homotopy commutes:
  \begin{equation}
 \begin{CD}
 P \times_G G^{-TG} @>\tau_\mu >> P \times_G G^{-TG} \wedge_X  P \times_G  G^{-TG}  \\
@V\bar \alpha V\simeq V     @V\simeq V\bar \alpha V \\
 \cd (\sxp)  @>> \mu^\vee >   \cd (\sxp) \wedge_X  \cd (\sxp).
\end{CD}
  \end{equation}
  
  In other words,  the induced map  
  $$\alpha : (P^{Ad})^{-\tv}   \xr{\simeq}  \cs^P_\bullet (X)  $$
respects coproducts up to homotopy.
  
  This completes the proof of parts (1) and (2) of Theorem \ref{main}.  Part (3) of Theorem \ref{main} follows from part (1) and  the Thom isomorphism applied to the vertical tangent bundle $T_{vert} \to \pad$.     The algebra structure on $H^*(\pad)$ was discovered first by Gruher   in \cite{gruher}.  The main point of part (1) of Theorem \ref{main} is that it realizes the work of Gruher   on the level of parameterized spectra and the induced twisted homology theory.   In \cite{gruher}  it was shown that in the case of the universal bundle $G \to EG \to BG$,  the algebra structure on $H^*(\pad \simeq LBG)$  (or equivalently the coalgebra structure on $H_*(LBG)$) is isomorphic to the Lie group string topology algebra of Chataur-Menichi \cite{chataurmenichi}.  This completes the proof of Theorem \ref{main}.

\section{ Twisted, compact Calabi-Yau ring spectra and the duality between manifold and Lie group string topology}  

The goal of this section is to study  duality phenomena in the  string topology  of a principal bundle $G \to P \to M$, where $G$ is a compact, $d$-dimensional Lie group, and  $M$ is a closed, $n$-dimensional manifold.  More specifically, our goal is to study   the duality between  the manifold string topology and the Lie group string topology in this setting.    To do this we describe the notion of ``twisted, compact Calabi-Yau ring spectra" and show how the string topology of such a principal bundle has this structure.  This notion is a lifting to the category of spectra, of the notion of  ``Calabi-Yau" algebras and categories as defined by Costello \cite{costello}, Kontsevich and his collaborators \cite{ks} \cite{ kv}, Lurie \cite{lurie}, and the author and Ganatra \cite{cg}.  

Our first result   is the following: 

\begin{theorem}\label{SWdual} For a principal bundle $G \to P \to M$ where $G$ a compact Lie group and $M$ is a closed manifold, the manifold string topology spectrum
$\cs_P^\bullet (M)$ and the Lie group string topology spectrum $\cs^P_\bullet (M)$ are Spanier-Whitehead dual.  Under this duality the ring spectrum structure of $\cs_P^\bullet (M)$ corresponds to the coalgebra structure of $\cs^P_\bullet (M)$.
\end{theorem}

\begin{proof} Recall from
  \cite{cjgauge} (and restated in the introduction above), the manifold string topology $\cs_P^\bullet$  is the twisted cohomology theory  corresponding to the    fiberwise suspension spectrum
$\Sigma^\infty (G_+) \to \Sigma^\infty_M (P^{Ad}_+) \to M$.    Using  particular a version of Poincar\'e duality proven in  by Klein in \cite{klein} (called ``Atiyah duality" in this paper), in \cite{CohenKlein} the first author and  Klein showed that
$$
\cs_P^\bullet (M) = \Gamma_M(\smp) \simeq (P^{Ad})^{-TM},
$$
and the ring structure comes from a generalized cup product in this (twisted) cohomology theory arising from the fiberwise ring structure of this parameterized spectrum.

 Furthermore Theorem \ref{main} above states that the Lie group string topology  $\cs^P_\bullet$ is the twisted homology theory corresponding to the fiberwise Spanier-Whitehead dual spectrum, $G^\vee  \to \cd(\smp) \to M$.  It was also shown that this is   a fiberwise coalgebra spectrum  whose coalgebra structure is (fiberwise) Spanier-Whitehead dual to the ring structure of the parameterized spectrum $ \Sigma^\infty (G_+) \to \Sigma^\infty_M (P^{Ad}_+) \to M$.  Finally it was shown that 
 there is a coproduct-preserving equivalence of   spectra,
$$
 (P^{Ad})^{-\tv}   \xr{\simeq}  \cs^P_\bullet (X) .
$$

 We remark that the fact that the Thom spectra  $(P^{Ad})^{-TM} \simeq \cs_P^\bullet (M)$ and  $(P^{Ad})^{-\tv} \simeq \cs^P_\bullet (M)$ are Spanier-Whitehead dual follows from classical Atiyah duality \cite{atiyahdual}.  
 
 This completes the proof.

\end{proof}

\med
\noindent \bf Remark.  \rm  When  $M$ is oriented, one can apply the two Thom isomorphisms,  $$H_*(\pad) \cong H_{*-n}((P^{Ad})^{-TM}) \quad \text{ and} \quad   H_*(\pad) \cong H_{*-d}((P^{Ad})^{-\tv}).$$   The Spanier-Whitehead duality above then yields a Frobenius algebra structure on $H_*(\pad)$ as discovered by Gruher \cite{gruher}.

\bg

We now strengthen this result by proving that in this situation, i.e a principal bundle $G \to P \to M$, where $G$ is a compact Lie group and $M$ is a closed, smooth manifold, that the spectrum $(\pad)^{-TM}$ is a ``twisted, compact Calabi-Yau ring spectrum".  The notions of Calabi-Yau differential graded algebras or $A_\infty$ algebras or (higher) categories were introduced in \cite{costello}, \cite{lurie}, \cite{ks}, \cite{kv} because of their connections with two-dimensional topological field theories.     This notion can be viewed as a derived version of a Frobenius algebra.  This will be made precise in Proposition \ref{cCYFrob} below.   In this paper we lift these ideas to the category of spectra, where we must deal with ``twisted" versions of these notions in order to get many interesting examples.  We actually introduce two versions of twisted Calabi-Yau ring spectra: a compact version and a smooth version.  This follows the ideas of Kontsevich and his collaborators \cite{ks}, and \cite{kv}, who worked with $A_\infty$ algebras  over a field of characteristic zero,
and of the author and Ganatra  \cite{cg} who  worked with $A_\infty$-algebras or categories over arbitrary fields. 

\med
We begin with the notion of a \sl ``twisted, compact, Calabi-Yau" \rm  ring spectrum.

\bg
\med
Recall that a compact $E_1$- ring spectrum $R$ is one that   is perfect as an $\bs$-module.    

\begin{definition}
A \bf ``twisted, compact Calabi-Yau ring spectrum" \sl (twisted  $cCY$)  \rm of dimension $n$
is a triple $(R, Q,  t)$, where $R$ 
is a compact ring spectrum,  
$Q$ is an $R$-bimodule which is  compact as an $\bs$-module,  and has the same $\bz/2$-homology as $R$:  $$H_*(Q; \bz/2) \cong H_*(R; \bz/2).$$  We refer to $Q$ as the ``twisting" bimodule.   If $Q = R$ we say that $R$ has \sl trivial \rm twisting.
$$t : THH(R; Q)  \to \Sigma^{-n}\bs$$ is a map of   spectra we  call the ``$n$-dimensional \sl trace map" \rm  that has the following duality property:  The pairing defined by the composition
$$
\langle \, , \, \rangle : R \wedge Q  \xr{\mu} Q \hk THH (R; Q) \xr{t} \Sigma^{-n}\bs
$$
is nondegenerate in the sense that the adjoint $R \to \Sigma^{-n}Q^\vee$ is an equivalence of $R$-bimodule spectra.  Here $\mu : R\wedge Q \to Q$ is the module structure, $Q \hk THH (R; Q)$ is the inclusion of the spectrum of zero simplices, and  $Q^\vee$ is the Spanier-Whitehead dual of $Q$, which exists because of the compactness assumption.      
\end{definition}

\med
The following observation is an immediate consequence of the definition.

\begin{proposition}\label{coalg}
Let $(R, Q, t)$ be a twisted compact Calabi-Yau ring spectrum.  Then the duality between $R$ and $Q$ defined by the nondegenerate pairing $\langle \, , \, \rangle $  defines a   coalgebra  structure on the twisting bimodule $Q$,  whose co-product is Spanier-Whitehead dual to the product in the ring structure $R$.
\end{proposition}

\med
The main applications of compact Calabi-Yau ring spectra  occur in the presence of orientations. We now define what we mean by this.

\med

\begin{definition} \rm  Let $(R, Q, t)$ be a twisted, $cCY$  ring spectrum of dimension $n$,  and let $E$ be a ring spectrum representing a homology theory $E_*$.  An   \sl ``$E_*$-orientation"  \rm of $(R, Q, t)$ is a pair $(u, \tilde t_E)$, where
$$
u : Q  \wedge E  \xr{\simeq}  R \wedge E
$$ 
is an equivalence of  the $E_*$-homology spectra as $R$-bimodules. Here $R$ acts trivially on $E$.      $$\tilde t_E :  THH(R  \,  R)_{hS^1}\wedge E \to \Sigma^{-n} E $$ is an $E$-module map from the homotopy orbit spectrum of the $S^1$-action induced by the cyclic structure, which factorizes the trace map $t$ in $E$-homology.  That is, the induced trace map   $t_E = t \wedge 1 : THH(R; \, Q)\wedge E  \to \Sigma^{-n}\bs\wedge E$  is homotopic to the composition

\begin{align}
t_E : THH(R; \, Q)\wedge E \xr{u}    THH(R ; \, R)\wedge E  &\xr{project} THH(R; \, R)_{hS^1}\wedge E \\ 
 &\xr{\tilde t_E}  \Sigma^{-n}E.
\end{align}
 
\end{definition}

\med
When $E =Hk$, the Eilenberg-MacLane spectrum for a field $k$, then a twisted, compact Calabi-Yau ring spectrum $(R, Q, t)$ together with an  $Hk$-orientation $(u, \tilde t_{Hk})$ defines a compact Calabi-Yau algebra structure on the singular chains with $k$-coefficients, $C_*(R; k)$, as defined in \cite{kv} and in \cite{cg}.

\med
The following gives a precise relation between twisted $cCY$-ring spectra and Frobenius algebras.

\med
\begin{proposition}\label{cCYFrob}  Let $(R, Q, t) $ be a twisted $cCY$ ring spectrum of dimension $n$, and let $E$ be a ring spectrum representing a homology theory $E_*$  with respect to which $(R, Q, t)$ has orientation $(u, \tilde t_E)$.  Then $R\wedge E$ is a Frobenius algebra over $E$ of dimension $n$.  That is, the pairing
\begin{align}
\langle \, , \, \rangle : (R\wedge E) \wedge (R\wedge E)  &\xr{multiply} R\wedge E \xr{\iota} THH(R\wedge E, R\wedge E) \notag \\
&\xr{project}THH(R\wedge E, R\wedge E)_{hS^1} \xr{\tilde t_E} \Sigma^{-n}E
\end{align}
is a nondegenerate pairing of $E-modules$.    Here $\iota : R\wedge E \hk THH(R\wedge E, R\wedge E)$ is the inclusion of the spectrum of $0-simplices$.   ``Nondegeneracy" means that the adjoint of this pairing,
$$
R\wedge E \to Rhom_E(R\wedge E, \Sigma^{-n}E)
$$
is an equivalence of $E$-modules.
\end{proposition}

\begin{proof} It is easily checked from the definition of orientation that the  pairing $\langle \, , \, \rangle$ defined above is homotopic to the composition
$$
 (R\wedge E) \wedge (R\wedge E) \xr {1\wedge u^{-1}}  (R\wedge E) \wedge (Q\wedge E) \xr{\mu} Q \wedge E \hk THH (R; Q)\wedge E \xr{t_E = t\wedge 1} \Sigma^{-n}\bs \wedge E.
 $$
 But this pairing is nondegenerate by the definition of the twisted Calabi-Yau structure.
\end{proof}

\med
We now give two important examples of twisted $cCY$ ring spectra.  

\med
\noindent \bf Example 1.    \rm  The first example shows how ordinary Poincar\'e or Atiyah duality fits the definition of twisted compact Calabi-Yau.  

\med
\begin{proposition}
Let $M$ be a closed $n$-dimensional manifold.  Then its Spanier-Whitehead dual, $M^\vee$, which, by Atiyah duality is equivalent to $\mtm$,   comes naturally equipped with the structure of a twisted $cCY$ ring spectrum of dimension $n$. 
\end{proposition}

 \begin{proof}  The suspension spectrum $\Sigma^\infty (M_+)$ can be viewed as a $M^\vee$ bimodule in the usual way.  Notice that since, by Atiyah duality,  $M^\vee$ is equivalent to $\mtm$, then the Thom isomorphism gives
$$
H_*(\Sigma^\infty (M_+); \bz/2) \cong H_{*-n}(M^\vee ; \bz/2).  
$$
So we let $R = M^\vee$ and let the twisting bimodule $Q = \Sigma^{-n}\Sigma^\infty (M_+)$, which we simply denote $\Sigma^{-n}(M_+)$. 

In order to define the $n$-dimensional  trace map on $THH(R; Q)$, we   first study its homotopy type.  This is a simplicial spectrum of finite type.  That is, for each $k$, the spectrum of $k$-simplicies is a finite spectrum.  For such a simplicial spectrum $\bx_\bullet$ we define its Spanier-Whitehead dual $\bx^\vee$ to be the totalization of the cosimplicial spectrum whose spectrum of $k$-simplices is the Spanier-Whitehead dual $\bx_k^\vee = Map (\bx_k, \bs).$   We then have the following result.
\med
\begin{lemma}  For $M$ a closed $n$-manifold, $R = M^\vee$ and $Q = \Sigma^{-n}(M_+)$, then  the Spanier-Whitehead dual of $THH(R; Q)$
 is given by
 $$
 THH(R; Q)^\vee \simeq \Sigma^n \ltm.
 $$
 \end{lemma}

 \med
 
 \begin{proof}  Note that
 $$
 THH(R;Q)_k = R^{(k)} \wedge Q = (M^k)^\vee \wedge \Sigma^\infty(M_+)\wedge S^{-n}.    
  $$
 Therefore in the cosimplicial spectrum $THH(R; Q)^\vee$, the spectrum of $k$-simplices is given by
 $$
 THH(R;Q)^\vee_k = \Sigma^\infty (M^k_+)\wedge M^\vee \wedge S^{n}.
 $$
 The coface maps are determined by the coalgebra structure of $\Sigma^\infty (M_+)$ defined by the diagonal map of $M$, as well as the bi-comodule structure of $M^\vee$, which up to homotopy can be described by   the maps
 $$
 M^\vee \simeq \mtm \to M_+ \wedge \mtm  \simeq M_+ \wedge M^\vee  \quad \text{and} \quad M^\vee \simeq \mtm \to \mtm \wedge M_+ = M^\vee \wedge M_+.
 $$
 These maps are   the  maps of Thom spectra induced by the diagonal $M \to M \times M$.   This cosimplicial spectrum is the $n$-fold suspension of the cosimplicial spectrum studied in \cite{cohenjones} where it was shown to have totalization equivalent to $\ltm$. 
 \end{proof}

 \med
\noindent {\bf Remark.} Notice that the inclusion of the spectrum of zero simplices,
$$
\Sigma^{-n}(M_+) \hk THH(R; Q)
$$
is Spanier - Whitehead dual to the map
$$
\Sigma^n \ltm \xr{eval.} \Sigma^n\mtm \simeq \Sigma^n M^\vee
$$ 
induced on Thom spectra by the usual evaluation fibration $LM \to M$.

 \med
 One way of thinking of the $n$-dimensional trace map $t : THH(R; Q) \to \Sigma^{-n}\bs$ is that it is Spanier-Whitehead dual to the $n$-fold suspension of the unit map in the ring structure of $\ltm$:
 $$
 \Sigma^n \bs \to\Sigma^n \ltm.
 $$
 More concretely notice that the augmentation map of $R$,
$$\eps : R = M^\vee \to \bs$$ and the map induced by sending all of $M$ to the non-base point  $$p: \Sigma^{-n}(M_+) \to \Sigma^{-n}\bs$$ define a map 
$$
t : THH(R; Q) = THH (M^\vee; \Sigma^{-n}(M_+)) \xr{(\eps, \, p)} THH(\bs; \Sigma^{-n}\bs) = \Sigma^{-n}\bs.
$$
The reader can now check that the composition
$$
M^\vee \wedge \Sigma^{-n}(M_+) \xr{\mu} \Sigma^{-n}(M_+) \to THH (M^\vee; \Sigma^{-n} (M_+)) \xr{t}  \Sigma^{-n}\bs
$$
is simply the $n$-fold desuspension of the duality map, and therefore is nondegenerate.  This proves that $(M^\vee, \Sigma^{-n}(M_+), t)$ is twisted, compact Calabi-Yau ring spectrum of dimension $n$.  \end{proof}

\med
We now consider orientations.  Let $E$ be any ring spectrum representing a  generalized homology theory  with respect to which $M$ is oriented.  The Thom isomorphism then defines
an equivalence

$$
u : \Sigma^{-n}(M_+)  \wedge E  \xr{\simeq}  \mtm \wedge E \simeq M^\vee \wedge E
$$ which is clearly an equivalence of $M^\vee$-bimodules.
Again consider the augmentation map $\eps : M^\vee \to \bs$.  Now  the orientation induces a Thom class map $\tau : M^\vee \simeq \mtm \to \Sigma^{-n}E$.  These maps  define a composition  
$$
 \tilde t_E:  THH(M^\vee; M^\vee)_{hS^1}\wedge E \xr{(\eps, \tau)} THH(\bs; \bs)_{hS^1}\wedge  \Sigma^{-n}E  \simeq \Sigma^{-n}(BS^1_+)\wedge E  \xr{p\wedge 1} \Sigma^{-n} E
 $$
 where $p :   BS^1_+ \to S^0$ is    the projection map.

We leave it to the reader to check that the composition 
$$
THH (M^\vee, \Sigma^{-n} (M_+))\wedge E \xr{u}  THH (M^\vee, M^\vee) \wedge E \xr{projection} THH (M^\vee, M^\vee)_{hS^1}  \wedge E  \xr{\tilde t_E} \Sigma^{-n}E
$$
is equivalent to $t\wedge 1 : THH (M^\vee, \Sigma^\infty (M_+))\wedge E  \to  \Sigma^{-n}\wedge \bs \wedge E$.   This proves that
 the pair $(u, \tilde t_E)$  defines an orientation of the twisted $cCY$ structure on $M^\vee$ with respect to $E$.  

\med
\noindent {\bf Remark.}     The above discussion  together with Proposition \ref{cCYFrob} implies that if $M^n$ is an oriented closed manifold, $M^\vee \wedge H\bz$ is a Frobenius algebra over the Eilenberg-MacLane spectrum $H\bz$.  Using the Atiyah duality equivalence
$M^\vee \simeq \mtm$ we see that $\mtm \wedge H\bz \simeq \Sigma^{-n}(M_+ \wedge H\bz)$ is a Frobenius algebra.  The multiplication
reflects the classical intersection product on the level of chains,  $C_{*+n}(M; \bz)$.  The comultiplication comes from the diagonal, $M \to M\times M$.

\med
\noindent \bf Example 2.  \rm The following example supplies the main ingredient for the proof of Theorem \ref{Frob} as stated in the introduction.

\med
\begin{proposition}\label{ptmcy}
Let $G \to P \to M$ be a principal bundle where $G$ is a compact Lie group of dimension $d$ and $M$ is a closed manifold of dimension $n$.     Then the manifold string topology ring spectrum $R = \cs^\bullet_P(M) \simeq (\pad)^{-TM}$ naturally admits the structure of  a twisted, compact Calabi-Yau ring spectrum of dimension $n-d$. 
\end{proposition}

\begin{proof}
We need to produce the twisting module $Q$ and a trace map $t : THH (R; Q) \to \Sigma^{d-n}\bs$.     For the twisting module we take the Lie group string topology spectrum $Q = \Sigma^{d-n}\cs^P_\bullet (M)  \simeq \Sigma^{d-n}(\pad)^{-\tv}$.  The fact that $R$ and $Q$   have isomorphic mod-$2$ homology follows from the Thom isomorphism.   The fact that $Q$ is indeed an $R$-bimodule follows from the Spanier-Whitehead duality of $R =  \cs^\bullet_P(M) $ and $\Sigma^{n-d}Q = \cs^P_\bullet (M)$ established in  Theorem \ref{SWdual}, reflecting Gruher's work \cite{gruher}.  The bimodule structure of $Q$ over $R$ is then the dual of the bimodule structure of $R$ over itself.

Notice also that $R = \cs^\bullet_P (M) \simeq  (\pad)^{-TM}$ has an augmentation $\eps : R \to \bs$.  To see this consider the following diagram, which we view as a map of principal bundles.  The right vertical sequence is thought of as a principal bundle with trivial group.
$$
\begin{CD}
G   @>>>   \{id \}  \\
@VVV    @VVV \\
P @>>> M \\
@VVV    @VV=V \\
M    @>>=>  M
\end{CD}
$$

This defines a map of twisted cohomology  ring spectra
$$
\cs_P^\bullet (M)    \la  \cs_M^\bullet (M)
$$ or equivalently,
$$  
\Gamma_M(\Sigma^\infty_M(\pad_+))    \la Map(\Sigma^\infty (M_+), \bs) = M^\vee. 
$$
The augmentation is then given by
$$
\eps : R = \cs_P^\bullet (M) \to M^\vee \to \bs
$$
 where the second map in this composition is  the augmentation of $M^\vee \to \bs$.  
 
 Notice that the above diagram also defines a map of bimodules,
     $$
   Q =  \Sigma^{d-n}\cs^P_\bullet (M)  \simeq \Sigma^{d-n}(\pad)^{-\tv}     \to  \Sigma^{d-n}\cs^M_\bullet (M) = \Sigma^{d-n} (M_+).
    $$
    Composing this map with the projection $p : \Sigma^{d-n}(M_+) \to \Sigma^{d-n} \bs$
    defines a map $u :  Q =  \Sigma^{d-n} \cs^P_\bullet (M)  \to \Sigma^{d-n}\bs$.       Putting these maps together gives a map of topological Hochschild homologies,
    $$
    t : THH(R, Q) \xr{(\eps, u)} THH(\bs, \Sigma^{d-n}\bs) = \Sigma^{d-n}\bs.
    $$

    We leave it to the reader to verify that the pairing defined by the  composition
 \begin{equation}\label{pairing}
  \langle \, , \, \rangle :   R\wedge Q = \cs_P^\bullet (M) \wedge \Sigma^{d-n}\cs^P_\bullet (M) \xr{\mu} Q = \Sigma^{d-n}\cs^P_\bullet (M) \hk THH(R, Q) \xr{t} \Sigma^{d-n}\bs
    \end{equation}
    is the duality map given by Theorem \ref{SWdual} above.  It is therefore nondegenerate.  This proves that the triple
    $(\cs_P^\bullet (M) ,  \Sigma^{d-n}\cs^P_\bullet (M), t)$ is a twisted compact Calabi-Yau ring spectrum of dimension $n-d$. 
    
    \end{proof}
    
    \med
    Notice that Propositions \ref{ptmcy}, \ref{coalg}, and \ref{cCYFrob} imply both Corollary \ref{frobenius} and Theorem \ref{Frob} as stated in the introduction.

\section{Gauge symmetry}

In this section we continue considering a principal bundle $G \to P \xr{p} M$, where $G$ is a compact Lie group of dimension $d$, and $M^n$ is a closed manifold of dimension $n$. 

Recall that the \sl gauge group \rm $\cg (P)$ of the bundle $P$ is the group of $G$-equivariant bundle automorphisms of $P$ living over the identity of $M$.   Said another way,  let $G \to \ca ut^G (P) \to M$ be the fibration whose fiber over $x \in M$ is the group of $G$-equivariant  automorphisms of the fiber $p^{-1}(x)$.   This bundle is a fiberwise group, and the gauge group is the group of sections
$$
\cg (P) = \Gamma_M(\ca ut^G(P)).
$$
Now a standard exercise shows that the bundle $\ca ut^G(P)$ is isomorphic to the adjoint bundle $G \to \pad \to M$.  Thus we may identify
$$
\cg (P) = \Gamma_M(\pad).
$$

In \cite{cjgauge} a fiberwise stabilization map was defined and studied:
\begin{equation}\label{stab1}
\rho : \Sigma^\infty (\cg (P)_+) = \Sigma^\infty (\Gamma_M(\pad)) \to \Gamma_M(\Sigma^\infty(\pad_+)) \simeq (\pad)^{-TM} = \cs^\bullet_P(M).
\end{equation}
$\rho$ is a map of ring spectra and also defines a map to the group of units of the (manifold) string topology ring spectrum
\begin{equation}\label{stab2}
\rho : \cg (P) \to GL_1(\cs^\bullet_P(M)).
\end{equation}
In \cite{cjgauge} this map was studied and computed in several important cases.  Now recall from Proposition \ref{ptmcy} that in the twisted compact Calabi-Yau structure of $ \cs^\bullet_P(M)$, the twisting bimodule is given by a suspension of the Lie group string topology spectrum,  $Q = \Sigma^{d-n}\cs^P_\bullet (M)  \simeq \Sigma^{d-n}(\pad)^{-\tv}$.  In particular   the Lie group string topology spectrum  $\cs^P_\bullet (M)$  inherits a coalgebra structure.  One of the  goals of this section is to show that there is a similarly defined and compatible gauge symmetry on this spectrum.  We also show how these actions are related, and describe different perspectives on this action.   We then compute two  examples of this gauge symmetry.

\med
\begin{theorem}\label{gaugesym} The twisting bimodule structure on the Lie group string topology spectrum,  $Q = \Sigma^{d-n}\cs^P_\bullet (M) $ has a natural action of the gauge group $\cg (P)$. That is, $\Sigma^{d-n}\cs^P_\bullet (M) $ is a module spectrum over the ring spectrum $\Sigma^\infty (\cg (P)_+)$.  Furthermore this action is compatible with the gauge symmetry on the manifold string topology spectrum $R = \cs^\bullet_P (M)$ via its twisted Calabi-Yau duality  pairing
$$
\langle \, , \, \rangle : R\wedge Q = \cs^\bullet_P (M) \wedge \Sigma^{d-n}\cs^P_\bullet (M)  \to \Sigma^{d-n}\bs
$$
as defined in the proof of Proposition \ref{ptmcy} (see Equation \ref{pairing}). That is, the adjoint equivalence
$$
 \cs^\bullet_P (M)  \xr{\simeq} \cs^P_\bullet (M)^\vee
$$
is equivariant with respect to the gauge  symmetry of  these spectra.
\end{theorem}
\begin{proof}  Recall that $\cs^P_\bullet (M)$ is the generalized homology associated to the parameterized spectrum
$$
G^\vee \to \cd (\Sigma^\infty_M(\pad_+)) \to M.
$$
We may take $\cd (\Sigma^\infty_M(\pad_+))$ to be the parameterized spectrum whose $k^{th}$ space over $M$ is the fibration
$$
Map (G, S^k) \to P \times_G Map (G, S^k) \to M
$$
where the action of $G$ on $Map (G, S^k)$ is the dual of the adjoint action as described in (\ref{gact}). This is because the spaces $Map(G,S^k)$ with this action form the underlying naive $G$-spectrum of $G^\vee$, on which the homotopy theory of $G^\vee$ as a $\Sigma^\infty(G_+)$-module is determined.

This fibration has a canonical section
$$ \sigma : M = P \times_G point = P \times_G \eps \hk P \times_G Map (G, S^k).$$ Here $\eps : G \to S^k$ is the constant map at the basepoint
$(1, 0, \cdots, 0) \in S^k$.  Then the $k^{th}$-space of the generalized homology spectrum 
$$
\cs^P_\bullet (M)   =  \cd (\Sigma^\infty_M(\pad_+)) / \sigma (M)
$$ is given by
$$
P_+\wedge_G Map (G, S^k).
$$
The structure maps are given by
$$
\Sigma (P_+\wedge_G Map (G, S^k)) = P_+\wedge_G \Sigma(Map (G, S^k))  \xr{1 \wedge s} P_+\wedge_G Map (G, S^{k+1})
$$
where $s : \Sigma(Map(G, S^k)) \to Map(G, \Sigma S^k)$ is given by $s(t, \phi)(g) = t\wedge\phi (g)$.

Now   the bundle $p : P \times_G Map (G, S^k) \to M$  is $\cg (P)$-equivariant with respect to the following action.   Let $\phi \in \cg (P) = \Gamma_M(\pad)$, and let $(y, \theta) \in P \times Map (G, S^k)$ represent an element in  $P \times_G Map (G, S^k)$.  Then
\begin{equation}\label{act}
\phi \cdot (y, \theta) = (y, h\cdot \theta)
\end{equation}
where $h \in G$ is the unique element so that $\phi (p(y, \theta)) \in P \times_G G^{Ad}$ is represented by $(y, h) \in P \times G$.    

The reader can check that this action is well-defined, and that the section $\sigma(M = P \times_G\eps)$ consist of fixed points of this action.
It therefore descends to a $\cg (P)$-action on $P_+\wedge_G Map (G, S^k)$.  

These actions (one for each $k$) clearly respect the structure maps and therefore defines an action of $\cg (P)$ on the spectrum $\cd (\Sigma^\infty_M(\pad_+)) / \sigma (M) = \cs^P_\bullet (M)$.   

Now as seen in Corollary \ref{frobenius} the Lie group string topology spectrum $\cs^P_\bullet (M)$ is Spanier-Whitehead dual to the manifold string topology spectrum $\cs_P^\bullet (M) = \Gamma_M(\Sigma^\infty_M(\pad_+))$. The action of the gauge group on this ring spectrum is given by the stabilization representation (\ref{stab1}), (\ref{stab2}), and it is immediate that the gauge symmetry defined on the Lie group string topology spectrum $\cs^P_\bullet (M)$ in (\ref{act}) is the dual action.  This implies that with respect to the twisted Calabi-Yau duality  pairing
$$
\langle \, , \, \rangle : R\wedge Q = \cs^\bullet_P (M) \wedge \Sigma^{d-n}\cs^P_\bullet (M)  \to \Sigma^{d-n}\bs
$$
then the corresponding adjoint equivalence, 
$$
 \cs^\bullet_P (M)  \xr{\simeq} \cs^P_\bullet (M)^\vee
$$
is equivariant with respect to the gauge  symmetry of  these spectra.
\end{proof}

\bg
We now study examples of this gauge symmetry and describe this symmetry from different perspectives.

\med
\noindent {\bf Example 1}.  Consider the $U(1)$ Hopf  bundle
$$
U(1) \to P= S^{2n+1} \to \bc\bp^n.
$$

Since $U(1)$ is abelian, the adjoint bundle $\pad$ is trivial, $U(1) \to \bc\bp^n \times U(1) \to \bc\bp^n.$  Therefore the gauge group is given by the mapping group,
$$
\cg (P) = Map (\bc\bp^n, U(1)).
$$

Also,  the fiberwise suspension spectrum $\Sigma^\infty (U(1)_+) \to \Sigma^\infty_{\bc\bp^n}(\pad_+) \to \bc\bp^n$ is given by the trivially parameterized spectrum
$$
\Sigma^\infty (U(1)_+) \to \bcp^n_+ \wedge \Sigma^\infty (U(1)_+) \to \bcp^n.
$$
(By the ``trivially parameterized spectrum   $ \bcp^n_+ \wedge \Sigma^\infty (U(1)_+)$" we mean the parameterized spectrum whose $k^{th}$ space is the trivial fibration $ \bcp^n \times \Sigma^k(U(1)_+) \to \bcp^n$.)

The twisted cohomology theory this parameterized spectrum represents is therefore actually untwisted, and so the defining spectrum of sections is the mapping spectrum,
$$
R = \cs^\bullet_P(\bcp^n) = Map(\Sigma^\infty(\bcp^n), \Sigma^\infty(U(1)_+)).
$$  This is an $E_\infty$- ring spectrum because the source of the mapping spectrum is an $E_\infty$- coalgebra spectrum and the target is an $E_\infty$-ring spectrum.

The action of the gauge group $\cg (P) = Map (\bc\bp^n, U(1))$ is then given by the map of ring spectra
\begin{equation}\label{stab}
\Sigma^\infty (\cg (P)_+) = \Sigma^\infty (Map (\bc\bp^n, U(1))_+) \xr{\sigma} Map(\Sigma^\infty(\bcp^n), \Sigma^\infty(U(1)_+)) = \cs^\bullet_P(\bcp^n)  = R
\end{equation}
where $\sigma$ is the obvious stabilization map.  The role of stabilization in understanding gauge symmetry on manifold string topology
spectra was studied in general in \cite{cjgauge}.   

Now consider the gauge symmetry on the bimodule $Q = \Sigma^{1-2n}\cs^P_\bullet(\bcp^n) = \Sigma^{1-2n}(S^{2n+1}_+\wedge_{U(1)}U(1)^\vee)$
where the action of $U(1)$ on the Spanier-Whitehead dual $U(1)^\vee$ is  (the dual of) the conjugation action.  Again, since $U(1)$ is abelian this action is trivial, so
 $$
 Q = \Sigma^{1-2n}(\bcp^n_+ \wedge U(1)^\vee).
 $$
 Of course $U(1) \cong S^1$, so $U(1)^\vee \simeq \Sigma^\infty (S^{-1} \vee S^0).$    By Spanier-Whitehead duality, the action of the gauge group $\cg (P)$ is given by composing the stabilzation map described above (\ref{stab})
 $$
 \sigma : \Sigma^\infty (\cg (P)_+) \to R
 $$
 with the $R$-bimodule action on the desuspension of its dual, $Q$, as described  in the proof of Proposition \ref{ptmcy}.

\bg
Before we move on to another example, we consider the action of the gauge group on the level of Thom spectra.
The point being that the Calabi-Yau ring spectrum in question is $R \simeq (\pad)^{-TM}$ and the twisting bimodule is $Q \simeq \Sigma^{d-n}(\pad)^{-\tv}$ are both Thom spectra.  To understand the induced gauge symmetry on these Thom spectra, we  first observe that the gauge group actually acts on the space $\pad$, and the actions on the Thom spectra are induced from it. 

Let $G \to P \xr{p} M$ be a principal bundle with $M$ a closed $n$-manifold and $G$ a compact Lie group of dimension $d$.  By abuse of notation we also call the projection map of the induced bundle $p : \pad \to M$.   Let  $\phi \in \cg (P) = \Gamma_M(\pad)$.  Since $\phi$ is a section of $\pad$, for $y \in \pad$,  $\phi (p(y))$ and $y$ live in the same fiber over $M$.  That is,
$p(\phi (p(y)) = p(y) \in M$.  Thus the pair, $(\phi (p(y)), y)$ lies in the fiber product $\pad \times_M \pad$.  Since $\pad$ is a fiberwise group, we can compose with the fiberwise multiplication $\mu : \pad \times_M \pad \to \pad$ to produce an element $\phi \cdot y = \mu (\phi (p(y)), y) \in \pad$.   The map
\begin{align}
\cg (P) \times \pad &\to \pad  \notag \\
(\phi, y) &\to \phi \cdot y  \notag
\end{align}
defines an action of the gauge group on $\pad$.  This in fact  defines a $\cg (P)$-equivariance on the fiber bundle,
$G \to \pad \to M$.  That is, the following diagram commutes:
\begin{equation}\label{commutes}
\begin{CD}
\cg (P) \times \pad   @>\cdot >>  \pad  \\
@VVV     @VVpV \\
M  @>>=> M
\end{CD}
\end{equation}
where the left vertical arrow composes the projection map $\cg (P) \times \pad \to \pad$ with the bundle map $\pad \to M$.
Therefore this action induces an action on any (virtual) vector bundle over $\pad$ that is pulled back from a bundle over $M$.
In particular, on the level of Thom spectra, there is an induced action
$$
\cg(P)_+ \wedge (\pad)^{-TM} \to (\pad)^{-TM}.
$$
This is easily seen to be equivalent to the action of $\cg (P)$ on $\cs^\bullet_P(M) \simeq (\pad)^{-TM}$ described above.
\med

We now  observe that the gauge symmetry on the Lie group string topology spectrum $\cs^P_\bullet (M) \simeq (\pad)^{-\tv}$ described above can also be viewed in terms of the space level   action of $\cg (P)$ on $\pad$.    This  is  a consequence of the following observation. 

\med
\begin{proposition}  Let $Act : \cg (P) \times \pad \to \pad$ be the action map described above.  Then there is an isomorphism of virtual bundles over $\cg (P) \times \pad$,
$$
\cg (P) \times -\tv \pad  \xr{\cong} Act^*(-\tv \pad).
$$
\end{proposition}

\begin{proof} We first observe  that the commutativity of diagram (\ref{commutes}) says  that there is an isomorphism of vector bundles over $\cg (P) \times \pad$,
$$
\cg (P) \times p^*(TM)  \cong Act^*(p^*(TM)).
$$
Notice also that  there is an isomorphism of  vector bundles, $D : \cg (P) \times T\pad \xr{\cong} Act^*(T\pad)$, where $T\pad \to \pad$ is the tangent bundle.  The isomorphism is given by differentiation of the action.  Now notice that there is an induced isomorphism of virtual bundles, which by abuse of notation we call $-D : \cg (P) \times -T\pad \xr{\cong} Act^*(-T\pad)$.  This is defined by the composition of isomorphisms
\begin{align}
Act^*(-T\pad) &= -Act^*(T\pad) \notag \\
&\cong -(\cg (P) \times T\pad) \quad \text{by the above}, \notag \\
&= \cg (P) \times -T\pad. \notag
\end{align}

Now, using the fact that
$$
-\tv\pad \cong -T\pad \oplus p^*(TM)
$$
we have that 
\begin{align}
Act^*(-\tv \pad) &\cong Act^*(-T\pad) \oplus Act^*(p^*(TM)) \notag \\
&\cong \cg (P) \times  \left(-T\pad \oplus p^*(TM)\right) \quad \text{by the above,} \notag \\
&\cong \cg (P) \times -\tv \pad.
\end{align}
\end{proof}

\med
The last explicit example of this gauge symmetry will be one that was studied initially in \cite{cjgauge}. 

\bg
\noindent {\bf Example 2.}  Consider the principal $SU(2)$- bundle over an oriented $4$-dimensional sphere,
$$
SU(2) \to P_k \to S^4
$$
having second Chern class $c_2(P_k) = k \in H^4(S^4) \cong \bz.$

In this case we restrict our attention to the \sl based \rm gauge group $\cg^b(P_k)$ which is defined to be the kernel of the homomorphism,
 \begin{align}
\cg(P_k) &\to SU(2) \notag \\
\phi &\to \phi (\infty). \notag
\end{align}
Here we are thinking of $S^4$ as the one point compactification, $\br^4 \cup \infty$.

\med
 In  the case $k=1$ the (based) gauge symmetry on the manifold string topology ring spectrum $\cs^\bullet_{P_k}(S^4)$ was studied in \cite{cjgauge}.  We now observe that the argument presented in \cite{cjgauge} quickly extends to $P_k$ for all $k$, and we then show
 how it gives an understanding of the gauge symmetry of the Lie group string topology spectrum $\cs^{P_k}_\bullet (S^4)$ as well.
 
 \med
 As has become standard notation, given a ring spectrum $R$, let $GL_1(R)$ denote the ``group of units" of $R$.  More precisely, $GL_1(R)$ is defined so that the following diagram  of spaces is homotopy cartesian:
 
 \med
\begin{equation}\label{units}
 \begin{CD}
 GL_1(R)    @>>>     \Omega^\infty (R) \\
 @VVV       @VVcomponents V \\
 \pi_0 (R)^\times        @>>\hk >  \pi_0(R)
  \end{CD}
\end{equation}
 Here $ \pi_0(R)$ is the discrete ring of components and $ \pi_0 (R)^\times$ is its group of units.

 \med
 In other words, $GL_1(R)$ consists of those path components of the zero space $\Omega^\infty (R)$ consisting of homotopy invertible elements.
 An action of a group $G$ on a ring spectrum $R$ (i.e a $\Sigma^\infty (G_+)$-module structure on $R$) is induced by an $A_\infty$-morphism (``representation") 
 $$
 \rho : G \to GL_1(R).
 $$  (See for example {\cite{lind}).   To understand the gauge symmetry on the manifold string topology spectrum $\cs^\bullet_{P_k}(S^4)$ we therefore want to describe the representation
\begin{equation}\label{rep}
 \cg (P_k) \to GL_1(\cs^\bullet_{P_k}(S^4)).
\end{equation}
 
 \med
 Now as observed in  \cite{cjgauge},
 the group-like monoid  $GL_1(\cs^\bullet_{P_k}(S^4))$   is equivalent to the grouplike monoid $hAut(\Sigma^\infty_{S^4}((P_k)_+)$  of homotopy automorphisms of the parameterized spectrum $\Sigma^\infty_{S^4}((P_k)_+)$.  
 
To understand this monoid of homotopy automorphisms, note that  given any ring spectrum $R$ and parameterized $R$-line bundle  $\ce$ over $M$,  there is a fibration sequence
\begin{equation}\label{based}
 hAut^b(\ce) \to hAut (\ce) \xr{ev} hAut^R(\ce_{x_0}) = GL_1(R)
\end{equation}
where  the map $ev$  evaluates an automorphism on the fiber over the basepoint  $x_0 \in M$.   The fiber is $ hAut^b(R)$, which is  the $A_\infty$ group-like monoid of based homotopy automorphisms.  This is the subgroup of  $hAut (\ce)$ consisting of those homotopy automorphisms that are equal to the identity  on the fiber spectrum  at the basepoint $\ce_{x_0}$.

Putting these facts together yields a fibration sequence of group-like monoids,
\begin{equation}\label{haut}
hAut^b(\Sigma^\infty_{S^4}((P_k)_+))  \to GL_1(\cs^\bullet_{P_k}(S^4)) \to GL_1(\Sigma^\infty (SU(2)_+)).
\end{equation}
As was done  in \cite{cjgauge}, we observe that since $SU(2) \cong S^3$,  then the defining diagram (\ref{units}) becomes, in the case of $R = \Sigma^\infty (SU(2)_+)$,
$$
 \begin{CD}
 GL_1(\Sigma^\infty (SU(2)_+))    @>>>   Q(S^3_+) \\
 @VVV       @VVcomponents V \\
\pm 1        @>>\hk >  \bz
  \end{CD}
$$
That is, $ GL_1(\Sigma^\infty (SU(2)_+))$ consists of two path components of the infinite loop space $Q(S^3_+) $ corresponding to the units
$\pm 1 \in \bz \cong \pi_0(Q(S^3_+))$.   We denote this space by $Q_{\pm 1}(S^3_+)$.     Therefore fibration (\ref{haut}) has base space $Q_{\pm 1}(S^3_+)$.  We now examine the homotopy type of the fiber, $hAut^b(\Sigma^\infty_{S^4}((P_k)_+))$.

By one of the main results of \cite{cjgauge} (Theorem 3), there is an equivalence
$$
hAut^b(\Sigma^\infty_{S^4}((P_k)_+)) \xr{\simeq} \Omega Map^b_k(S^4, BGL_1(\Sigma^\infty (SU(2)_+)))  = \Omega^4_k GL_1(\Sigma^\infty (SU(2)_+))
$$ where $Map^b_k$ denotes the path component of the based mapping space  corresponding to $k \in \bz = \pi_0(Map^b(S^4, BGL_1(\Sigma^\infty (SU(2)_+)))$.  Similarly $\Omega^4_k$ denotes the corresponding path component in $\Omega^4GL_1(\Sigma^\infty (SU(2)_+)).$   Now since $\Omega^4GL_1(\Sigma^\infty (SU(2)_+))$ is a group-like monoid, all of its path components are homotopy equivalent.
So we therefore have the following result, which gives a good understanding of the group of units of the manifold string topology spectrum of the principal bundle $SU(2) \to P_k \to S^4$.

\begin{lemma}\label{fibgauge}
For any $k$, there is an equivalence of group-like monoids, $ \phi_k : hAut^b(\Sigma^\infty_{S^4}((P_k)_+)) \xr{\simeq} \Omega^4Q(S^3_+).$  Furthermore there are homotopy fibration sequences of group-like monoids
$$
\Omega^4Q(S^3_+) \xr{\iota_k} GL_1(\cs^\bullet_{P_k}(S^4)) \xr{q_k} Q_{\pm 1}(S^3_+).
$$
\end{lemma}

\med
In order to understand the representation $\rho_k : \cg^b(P_k) \to GL_1(\cs^\bullet_{P_k}(S^4))$  describing the gauge symmetry of the manifold string topology spectrum, we now consider the homotopy type of the based gauge group $\cg^b(P_k)$. Again, for $k=1$ this was done in \cite{cjgauge}, and we simply adapt the argument there to apply to all $k$. 

By a basic result on the topology of gauge groups proved by Atiyah and Bott in \cite{atiyahbott}, we have that
$$
\cg^b(P_k) \simeq \Omega Map^b_k(S^4, BSU(2)),
$$
where, as above, $Map^b_k$ denotes the path component of degree $k$ based maps. This (based) loop space is equivalent to $\Omega \Omega^3_kSU(2) = \Omega \Omega^3_kS^3$.  Since $\Omega^3S^3$ is a group-like monoid, all of its path components are equivalent, so we have that the following.

\begin{lemma}
For any $k$, there is an equivalence of group-like monoids,  $\psi_k : \cg^b(P_k) \xr{\simeq} \Omega^4S^3$.  
\end{lemma}

\med
By Proposition 5 of \cite{cjgauge},  one knows that given any principal bundle over a manifold, $G \to P \to M$,  the action of the gauge group (and therefore the based gauge group) on the manifold string topology spectrum $\cs^\bullet_P(M)$ is defined by the representation given by the stabilization map
\begin{align}
\cg^b(P) &\xr{\rho}  GL_1(\cs^\bullet_P(M)) \notag \\
\Omega Map^b_P(M, BG)    &\xr{\sigma} \Omega Map^b_P (M, BGL_1(\Sigma^\infty (G_+)) \notag
\end{align}
where $\sigma$ is induced by the natural inclusion $G \hk GL_1(\Sigma^\infty (G_+)$. Here $Map^b_P$ denotes the path component of the based mapping space that  classifies the bundle $P$. 

In the case of $SU(2) \to P_k \to S^4$ then Lemma \ref{fibgauge} says that the representation $\rho_k : \cg^b(P_k) \to GL_1(\cs^\bullet_{P_k}(S^4))$ is given by the stabilization map
\begin{equation}\label{act}
\Omega^4S^3    \xr{\sigma}  \Omega^4Q(S^3_+)  \xr{\iota_k} GL_1(\cs^\bullet_{P_k}(S^4)).
\end{equation}
      where  $\sigma$ is induced by the map $ u_k : S^3 \to    Q(S^3_+) \simeq Q(S^3) \times QS^0$  that sends $S^3$ to a generator of $\pi_3Q(S^3) \cong \bz$  cross the basepoint of the component $Q_k(S^0)$.

\bg
Finally, notice that given any compact ring spectrum $R$, the group of units $GL_1(R)$ acts on its Spanier-Whitehead dual $R^\vee$ by the dual action of $GL_1(R)$ on $R$.   Given the Spanier-Whitehead duality between the manifold string topology spectrum $\cs^\bullet_{P_k}(S^4) \simeq (P_k^{Ad})^{-TS^4}$ and the Lie group string topology spectrum $\cs^{P_k}_\bullet(S^4) \simeq (P_k^{Ad})^{-\tv}$ then (\ref{act}) describes the action of the based gauge group $\cg^b(P_k)$ on the Lie group string topology spectrum as well.

\med
To end this section we point out that Proposition \ref{cCYFrob} and the above analysis of gauge symmetry implies the following.

\begin{theorem}\label{frobgauge}   Let $G \to P \to M$ be a principal bundle over a closed $n$-manifold $M$, with $G$ a $d$-dimensional compact Lie group.   Let $E$ be any ring spectrum with respect to which the compact Calabi-Yau structure on $\cs^\bullet_P(M)$ given in Proposition \ref{ptmcy} is oriented.  Then the homology $E_*(\cs^\bullet_P(M))$ is a Frobenius algebra over the homology of the gauge group, $E_*(\cg(P))$.  That is, the following conditions hold:
\begin{itemize}
\item The homology algebra structure of the manifold string topology  ring spectrum $E_*(\cs^\bullet_P(M))$ carries the structure of an algebra over $E_*(\cg (P)).$ \\
 \item The homology coalgebra structure of the Lie group string topology coalgebra spectrum $E_*(\cs_\bullet^P(M))$ is a module over $E_*(\cg (P))$, and \\
 \item The duality homomorphism defined by the Frobenius algebra structure induced from the compact Calabi-Yau structure, 
 $$
 E_*(\cs^\bullet_P(M))  \xr{\cong} E_{n-d-*}(\cs_\bullet^P(M))^*
 $$
 is an isomorphism of $E_*(\cg (P))$ modules.
 \end{itemize}
 \end{theorem}
    
    \section{Twisted, smooth Calabi-Yau ring spectra, Thom ring  spectra,  and Lagrangian immersions of spheres}

\bg

We now turn to the notion of twisted Calabi-Yau structures for \sl smooth \rm ring spectra. 

Recall that  a \sl smooth \rm ring spectrum $A$ is one that is perfect as an $A$-bimodule.  That is, it is perfect as a left $A\wedge A^{op}$- module. Given a smooth ring spectrum $A$, let $A^!$ be its ``bimodule dual".  That is,
$$
A^! = Rhom_{A \wedge A^{op}} (A, A \wedge A^{op}).
$$

Now recall that given any $A$-bimodules $P$ and $Q$ there  is a cap product pairing
$$
\cap : Rhom_{A \wedge A^{op}}(A, P) \quad \wedge  \quad  A \wedge^L_{A \wedge A^{op}}Q  \quad  \la \quad P \wedge^L_{A \wedge A^{op}} Q.
$$

When $P = A \wedge A^{op}$, then one can take a cap product with respect to a map $\rho : \bs \to A \wedge_{A \wedge A^{op}}Q $
to obtain a map
\begin{equation}\label{cap}
\cap \rho : A^! \to Q.
\end{equation}

\begin{definition}
A \bf ``twisted, smooth Calabi-Yau ring spectrum" (twisted  $sCY$)  \rm  of dimension $n$ 
is a triple $(A, P, \sigma)$, where $A$ 
is a smooth ring spectrum and 
$P$ is a smooth $A$-bimodule    and has the same $\bz/2$-homology as $A$:  $$H_*(P; \bz/2) \cong H_*(A; \bz/2).$$  We refer to $P$ as the ``twisting" bimodule.   If $P = A$ we say that $A$ has \sl trivial \rm twisting.
 $$\sigma : \Sigma^n \bs \to THH(A, \, P)$$ is a map of spectra we call the `` $n$-\sl dimensional  cotrace map" \rm  which  has the following duality property:
 The induced cap product pairing
$$
\cap \sigma : A^! = Rhom_{A \wedge A^{op}} (A, A \wedge A^{op})   \la  \Sigma^{-n} P  
$$
is an equivalence of $A$-bimodule spectra.  
\end{definition}

\noindent \bf Note: \rm Given a graded module $P$ over a ring $R$ let $P[-n]$ denote the desuspension $\Sigma^{-n}(P)$.  

\med

Like in the compact case, in most   applications a twisted, smooth Calabi-Yau spectrum is reduced over a homology theory with respect to which the twisting becomes trivialized, or ``oriented".   We now make this precise.

\begin{definition}  Let $(A, P, \sigma)$ be a twisted, $sCY$  ring spectrum of dimension $n$, and let $E$ be a ring spectrum representing a homology theory $E_*$.  An   \sl ``$E_*$-orientation"  \rm of $(A, P, \sigma)$ is a pair $(u, \tilde \sigma_E)$, where
$$
u :     P \wedge E  \xr{\simeq}  A \wedge E
$$ 
is an equivalence    of   $E$-module spectra  as $A$-bimodules. Here $A$ acts trivially on $E$.      $\tilde \sigma_E : \Sigma^{n} E \to THH(A , \,  A)^{hS^1}\wedge E $ is a map to the $E$-homology of the homotopy fixed point spectrum  of the $S^1$-action induced by the cyclic structure, which factorizes the cotrace map  $\sigma$ in $E_*$-homology.  That is, the following diagram homotopy commutes:

$$
\begin{CD}
\Sigma^nE  @>  \tilde \sigma_E >>   THH(A, \,  A)^{hS^1} \wedge E  \\
@V\sigma \wedge 1 VV    @VV j V \\
THH(A, \, P)\wedge E  @>\simeq >u_*>  THH(A, A)\wedge E.
\end{CD}
$$
Here $j$ is the  natural inclusion of the homotopy fixed points.   
\end{definition}

\med
Notice that when $E =Hk$, the Eilenberg-MacLane spectrum for a field $k$, then a twisted, smooth Calabi-Yau spectrum $(A, \tau, \sigma)$ together with an  $Hk$-orientation $(u, \tilde \sigma_{Hk})$ defines a smooth Calabi-Yau algebra structure on the singular chains with $k$-coefficients, $C_*(A; k)$, as defined in \cite{kv} and in \cite{cg}.

\subsection{Thom spectra of virtual bundles over the loop space of a manifold}

 We now consider important examples of a twisted, smooth Calabi-Yau spectra.  These are Thom ring spectra of virtual bundles over $\Omega M$, for $M$ a closed manifold.    

We begin by studying the Thom spectrum of  the trivial bundle over $\Omega M$, namely the suspension spectrum $\Sigma^\infty (\Omega M_+)$. The following  generalizes the chain complex analogue proven by the first author and Ganatra in \cite{cg}.
 
 \med
 \begin{theorem}\label{OmegaM}  Let $M$ be a closed manifold of dimension $n$.  Then the suspension spectrum of its based loop space,
 $\Sigma^\infty (\Omega M_+)$ can be given the structure of a twisted, smooth Calabi-Yau ring spectrum of dimension $n$. 
 \end{theorem}
 
 \begin{proof}  In order to give $ A = \Sigma^\infty (\Omega M_+)$ a twisted $sCY$ structure, we need to define a  twisting bimodule and a cotrace map.
 Consider the virtual bundle $-TM$ over $M$.  The associated virtual spherical fibration is classified by a map $B_{-\tau_M} : M \to BGL_1(\bs)$.
 By taking the loop of this map and applying suspension spectra we get a map of ring spectra
\begin{equation}\label{tau}
 -\tau_M : A = \slm  \to \Sigma^\infty (GL_1(\bs)).
\end{equation}
This defines a $\slm$-bimodule structure on the sphere spectrum.  We let $\bst$ be the sphere spectrum with this bimodule structure, and we define
 $P =   \bst \wedge \slm =  \bst \wedge A$ to be the induced bimodule.  Here $\bst \wedge \slm$ is given the diagonal $A$-bimodule structure. 
 
  $P$  will be the twisting bimodule. 
  To describe the $n$-dimensional cotrace map, we first need the following observation.  
 
 Replace $\Omega M$ by its Kan loop group, which by abuse of notation we still write as $\Omega M$.   Consider the adjoint action of $\Omega M$ on $\Omega M \times \Omega M$ defined by   
 $$
 g \cdot (h_1, h_2) = (gh_1, h_2g^{-1}).
 $$
  
If $E$ is the homotopy orbit space of this action, $E = pt \times^L_{\Omega M} (\Omega M \times \Omega M)^{Ad}$  then we have a  homotopy fiber sequence (i.e each successive three terms is a homotopy fibration)
\begin{equation}\label{Delta}
\Omega M \xr{\tilde \Delta} \Omega M \times \Omega M \to E \xr{\pi} M.
\end{equation}
 where $\tilde \Delta (g) = (g, g^{-1}).$  From this sequence it immediately follows that $E \simeq \Omega M$ and $\pi : E \to M$ is null homotopic.

 The Thom spectrum of the pull-back virtual bundle $\pi^*(-TM)$ which we denote by $E^{-TM}$ is given by
 $$
 \Sigma^{-n} \bs_{-\tau_M} \wedge^L_{\Omega M} \Sigma^\infty (\Omega M_+) \wedge  \Sigma^\infty (\Omega M_+)  =  \Sigma^{-n} \ \bs_{-\tau_M}\wedge^L_A (A\wedge A^{op})^{Ad}.
 $$
 However, since $\pi : E \to M$ is null homotopic, $\pi^*(-TM)$ is trivial, and there is an equivalence
\begin{equation}\label{etm}
 h : E^{-TM} = \Sigma^{-n} \ \bs_{-\tau_M}\wedge^L_A (A\wedge A^{op})^{Ad} \xr{\simeq}  \Sigma^{-n} \bs_{-\tau_M} \wedge A   = P[-n].     
\end{equation}
 Furthermore one can check that this equivalence can be taken to be one of $A$-bimodules. 
 
 We therefore have an equivalence
 $$
 \begin{CD}
   A \wedge^L_{A\wedge A^{op}} (\bs_{-\tau_M}\wedge_A (A\wedge A^{op})^{Ad}) @>h >\simeq >  A \wedge^L_{A\wedge A^{op}} (\bs_{-\tau_M}\wedge A)  @>= >>    = THH(A, \, P).
 \end{CD}
 $$
 
 Now the map $\tilde \Delta : \Omega M  \to \Omega M \times \Omega M$ defines a ring map on the level of suspension spectra,
 which by abuse of notation we still call $\tilde \Delta$,
 $$
 \tilde \Delta : A \to A \wedge A^{op}.
 $$
 We then get a change-of-rings equivalence,
 $$
\phi:  A^{Ad}   \wedge^L_A  \bs_{-\tau_M} \xr{\simeq}  A \wedge^L_{A\wedge A^{op}} ((A\wedge A^{op})^{Ad}\wedge_A^L \bs_{-\tau_M} ).
 $$

Consider the unit map $u : \bs \to \slm = A$.  This defines a map
 $$
 u : \bs \wedge^L_A    \bs_{-\tau_M}     \to A^{Ad}   \wedge^L_A  \bs_{-\tau_M}.
 $$
 
 Now  $ \bs \wedge^L_A    \bs_{-\tau_M} $ is the Thom spectrum $\Sigma^n(\mtm)$.   Thus there is a Pontrjagin-Thom map $\gamma : \Sigma^n \bs \to \Sigma^n(\mtm) =  \bs \wedge^L_A    \bs_{-\tau_M}$.  This can be viewed as the $n$-fold suspension of the unit map of the Spanier-Whitehead dual, $$\bs \to M^\vee \simeq \mtm.$$
 
 \med
 We can now define the $n$-dimensional cotrace map $\sigma : \Sigma^n \bs \to THH(A, \,  P) $ to be the composition
\begin{align}
 \sigma :  \Sigma^n \bs \xr{\gamma} \bs \wedge^L_A    \bs_{-\tau_M} \xr{u} A^{Ad}   \wedge^L_A  \bs_{-\tau_M} &\xr{\phi}A \wedge^L_{A\wedge A^{op}} ((A\wedge A^{op})^{Ad}\wedge_A^L \bs_{-\tau_M} ) \notag \\ &\xr{h}  A \wedge^L_{A\wedge A^{op}} (\bs_{-\tau_M}\wedge A) = THH(A, \, P) 
\end{align}

 \med
  To show $\sigma$  is a valid cotrace map we need to check that it satisfies the required duality condition.  Namely, we need to show that the cap product,
 \begin{equation}\label{cap}
 \cap \sigma :  A^! =  Rhom_{A \wedge A^{op}}(A, \, A \wedge A^{op})  \la \Sigma^{-n}\bs_{\negtm}\wedge A = P[-n]
 \end{equation} is an equivalence of $A$-bimodules.  This map was constructed to be an $A$-bimodule map, so it suffices to check that it is an ordinary  weak equivalence. 


 In order to do this, we study the homotopy type of $Rhom_{A \wedge A^{op}}(A, \, A \wedge A^{op})$, where $A = \slm$.   Notice that since $A$ is a connective Hopf algebra (in the weak sense),  we have an equivalence
 $$
 Rhom_{A \wedge A^{op}}(A, \, A \wedge A^{op})  \simeq Rhom_A (\bs, \,  (A \wedge A^{op})^{Ad}).
 $$
    
Consider again the homotopy fibration $\Omega M \times \Omega M \to E \to M$, and its fiberwise suspension spectrum. 
  This is the parameterized spectrum 
 $$
 \slm \wedge \slm \to \Sigma^\infty_M(E_+) \to M.
 $$
 The spectrum of sections of this parameterized spectrum  is the spectrum whose homotopy type we are trying to compute:
 
 $$
  \Gamma_M(\Sigma^\infty_M(E_+)) \simeq Rhom_A (\bs, \,  (A \wedge A^{op})^{Ad}) \simeq  Rhom_{A \wedge A^{op}}(A, \, A \wedge A^{op}).
 $$
 
 If we let $\ce^\bullet$ be the twisted equivariant cohomology theory represented by the parameterized spectrum $\Sigma^\infty_M(E_+)$,
 then what we are trying to compute is $\ce^\bullet (M)$.  But by twisted Poincar\'e duality theorem of Klein \cite{klein} as described in \cite{CohenKlein}, this twisted cohomology spectrum is given by the Thom spectrum
 $$
 \ce^\bullet (M) = \Gamma_M(\Sigma^\infty_M(E_+)) \simeq E^{-TM}.
 $$
 Now as seen in (\ref{etm}) above, there is an equivalence 
 $$
 h : E^{-TM} \xr{\simeq} \Sigma^{-n}\bs_{\negtm} \wedge A.
 $$
 Putting these together gives an equivalence,
 $$
 A^! =  Rhom_{A \wedge A^{op}}(A, \, A \wedge A^{op}) =  \ce^\bullet (M) \simeq E^{-TM} \xr{h} \Sigma^{-n}\bs_{\negtm} \wedge A = P[-n].
 $$
 
 We leave it to the reader to check that the cap product map
 $$
 \cap \sigma : A^! \to P[-n]
 $$
 induces such an equivalence.      
 
 Given this, the proof that $(\slm, \negtm, \sigma)$ is a twisted, smooth Calabi-Yau ring spectrum is complete.
  \end{proof}

Orientations on this twisted, smooth Calabi-Yau ring spectrum will be addressed in a more general context, in the setting of Thom spectra of virtual bundles over $\Omega M$.
 \med

 We now generalize the above to the setting of Thom ring spectra.  Namely, let 
 $\Omega f: \Omega M \to BGL_1(\bs)$ be a loop map. That is, it is obtained by applying the based loop functor to a   map $f : M \to B(BGL_1(\bs))$.   Let $\Omega M ^{\Omega f}$ denote the Thom spectrum of $\Omega f$. By a theorem of Lewis, $\Omega M ^{\Omega f}$ is a ring spectrum.   Let $E$ be any commutative ($E_\infty$) ring spectrum.  As is usual we say that a virtual bundle
 $\omega : X \to BGL_1(\bs)$ is $E$-orientable if there is a ``Thom class"  $\tau : X^\omega \to E$  so that the composition
 \begin{equation}\label{diag}
\theta_\tau :  X^\omega  \wedge E  \xr{\Delta \wedge 1}  X_+ \wedge X^\omega \wedge E \xr{1 \wedge \tau \wedge 1} X_+ \wedge E \wedge E \xr{1 \wedge mult.} X_+ \wedge E
\end{equation}
 is an equivalence (the ``Thom isomorphism").  Here $ \Delta : X^\omega \to X_+ \wedge X^\omega$ is the map of Thom spectra induced by the diagonal map $X \to X \times X$.  Again, as is usual we define an $E$-orientation of a manifold $M$ to be an $E$-orientation of its tangent bundle $\tau_M$  or equivalently of $-\tau_M$.   An $E$-orientation of a loop map   $\Omega f : \Omega Y \to BGL_1(\bs)$ has the additional requirement that the Thom class $\tau : (\Omega Y)^{\Omega f} \to E$ be a map of ring spectra.  In this case notice that the orientation equivalence $\theta_\tau$  in (\ref{diag}) is an equivalence of ring spectra.


\begin{theorem}\label{prop-thom-scy}
$A= \Omega M ^{\Omega f}$  naturally has the structure of a twisted sCY ring spectrum of dimension $n$. Furthermore,  suppose $E$ is a commutative ring spectrum.   Then an $E$ orientation $\tau : \mtm \to E$ of $M$ and an $E$-orientation $A =  \Omega M ^{\Omega f} \to E$   together  induce an $E$-orientation on the $sCY$ structure on $A$.     
 \end{theorem}

\noindent {\bf Remark. }
This theorem readily generalizes to the setting of generalized Thom spectra of maps to $BGL_1(R)$, where $R$ is a commutative ring spectrum.


\med
\begin{proof}
Let $\bst$ denote the sphere spectrum viewed as a    $\Sigma^\infty (\Omega M_+)$-module   as above.    
We first observe that by the twisted Poincar\'e duality theorem of Klein \cite{klein} as described in \cite{CohenKlein}  (see also \cite{DGI} and \cite{malm}), there is an equivalence of $\Sigma^\infty (\Omega M_+)$-modules

\begin{equation} \label{kleinPD}
Rhom_{\Sigma^\infty (\Omega M_+)} (\bs, \Sigma^\infty(\Omega M _+)) \simeq \Sigma^{-n} \bst
\end{equation}

Here $\Sigma^\infty (\Omega M_+)$ acts on  itself by left multiplication.   The reason this equivalence holds is the following.
The left-hand side describes the section spectrum  of the fiberwise suspension spectrum of the path-loop fibration, $\Omega M \to P(M) \to M$.     By Klein's theorem on Poincar\'e duality for parametrized spectra, the section spectrum (i.e twisted cohomology spectrum) associated to this parametrized spectrum is equivalent to a twisting by $-TM$ of the homology spectrum, 
$ 
 \Sigma^\infty (\Omega M_+)  \wedge^L_{\Sigma^\infty (\Omega M_+)}  \Sigma^{-n}\bst \simeq \Sigma^{-n} \bst.
$    This  equivalence is given by cap product with the Pontrjagin-Thom  class $t_M: \bs \to M^{-TM} = \Sigma^{-n} (\bs \wedge^L _{\Sigma^\infty (\Omega M_+)} \bst)$.

\med

Consider now the map of ring spectra $\Sigma^\infty  (\Omega M_+) \to A \wedge A^{op}$, induced on Thom spectra by the map $\tilde{\Delta}: \Omega M \to \Omega M \times \Omega M$ of (\ref{Delta}). The source of this map is indeed $\Sigma^\infty(\Omega M_+)$, as the composite

$$\xymatrix{
\Omega M \ar[r]^-{\tilde{\Delta}} & \Omega M \times \Omega M \ar[r]^-{\Omega f \times \Omega f} & BGL_1(\bs) \times BGL_1(\bs) \xr{multiply}   BGL_1(\bs)
}$$
is null homotopic.  The following is a result of the  second author (Theorem 5.1 of \cite{klang}).

\begin{theorem}\label{thm5.1}
Under this action of $\Sigma^\infty(\Omega M _+)$ on $A \wedge A^{op}$, there is an equivalence of $A \wedge A^{op}$-modules

\begin{equation}
A \simeq \bs \wedge^L _{\Sigma^\infty( \Omega M_+)} (A \wedge A^{op})
\end{equation}   Here the module structure on the right hand side is by right action on $A \wedge A^{op}$.
\end{theorem}

Notice that $\bs$ is a perfect $\Sigma^\infty(\Omega M_+)$-module since $M$ is assumed to be compact. That is, $\bs$ is equivalent to a retract of a finite  $\Sigma^\infty(\Omega M_+)$-module. Applying $(-) \wedge^L _{\Sigma^\infty(\Omega M_+)}(A \wedge A^{op})$, we  can conclude from Theorem \ref{thm5.1} that $A$ is a retract of a finite  $A \wedge A^{op}$-module, hence a perfect $A \wedge A^{op}$-module. That is, $A$ is a smooth ring spectrum. Also because $\bs$ is perfect as a  $\Sigma^\infty(\Omega M_+)$-module, we can apply $(-) \wedge^L _{\Sigma^\infty ( \Omega M_+)} (A \wedge A^{op})$ to both sides of  the equivalence (\ref{kleinPD}) to obtain

\begin{equation}\label{kleq}
Rhom_{\Sigma^\infty ( \Omega M_+)}(\bs, A \wedge A^{op}) \simeq \Sigma^{-n} (\bst \wedge^L _{\Sigma^\infty (\Omega M_+)} A \wedge A^{op}).
\end{equation}

We now take our twisting bimodule to be $P = \bst \wedge^L _{\Sigma^\infty ( \Omega M_+)} A \wedge A^{op}$. Notice that if the original map$  f$ is null homotopic, this agrees with the twisting bimodule given in the proof of Theorem \ref{OmegaM}. Furthermore, by the Thom isomorphism for $-TM$, 

$$H\mathbb{Z}/2 \wedge P \simeq H\mathbb{Z}/2 \wedge A.$$
This is an equivalence of $A$-bimodules because the equivalence in Theorem  \ref{thm5.1} is.    
Moreover, by the above theorem, $  Rhom_{\aop} (A, \aop)    \simeq Rhom_{\aop}((\aop) \wedge^L_{\Sigma^\infty (\Omega M_+)}\bs, \, \aop) \simeq    Rhom_{\Sigma^\infty (\Omega M_+)}(\bs, A \wedge A^{op}).$       So the  equivalence (\ref{kleq}) becomes 

\begin{equation}\label{bimoddual}A^! = Rhom_{A \wedge A^{op}}(A, A \wedge A^{op}) \simeq \Sigma^{-n} P. \end{equation}

Our goal is to show that this equivalence is given by taking the cap product with an appropriate $n$-dimensional cotrace map $\sigma : \Sigma^n\bs \to THH(A,P)$ which we now define.

\med
Since $A$ is smooth, 
\begin{align}\label{thhap}
 THH(A, P) &= \Sigma^n Rhom_{\aop}(A, \aop) \wedge^L_{\aop} A  \simeq  \Sigma^nRhom_{\aop} (A, A) \\
&\simeq \Sigma^nTHH^{\bullet}(A, A) \notag 
 \end{align}
where the last quantity is the topological Hochschild cohomology.  

The inverse equivalence also has a very natural description.  
Consider the  $\Sigma^\infty (\Omega M_+)$-module structure on $A =  (\Omega M)^{\Omega f}$ given by the generalized conjugation action defined to be the pull-back of the $A\wedge A^{op}$-action on $A$ to $\Sigma^\infty (\Omega M_+)$ via the ring map $\Sigma^\infty (\Omega M_+) \to \aop$ defined above.  This action was studied in detail by the second author in \cite{klang}.  In \cite{klang} the second author showed that there are equivalences, 
\begin{align}\label{ac}
THH (A, A) &\simeq \bs \wedge^L_{\sgm} A^c \\
THH^{\bullet}(A, A) &\simeq  Rhom_{\Sigma^\infty (\Omega M_+)} (\bs, A^c) \notag
\end{align}
where $A^c$ denotes the algebra $A$ with this generalized conjugation action of $\sgm$.  

Recall the cap product operation
\begin{equation}
 \left(Rhom_{\Sigma^\infty (\Omega M_+)} (\bs, A^c) \right) \wedge \left( \bs \wedge^L_{\sgm} \bst \right)  \xr{\cap} \bst \wedge^L_{\sgm} A^c \simeq THH(A, P).
 \end{equation}  
 Now $\bs \wedge^L_{\sgm} \bst  = \Sigma^n\mtm$ so there is a Pontrjagin-Thom map $S^n \xr{t_M} \Sigma^n\mtm$ which corresponds to the unit $\iota \in M^\vee$ under the Atiyah-equivalence of the Spanier-Whitehead dual of a manifold $M^\vee$ and its Thom spectrum $\mtm$.    We then get an induced equivalence
 
 \begin{align}\label{inverse}
 &THH^{\bullet}(A, A) \wedge S^n  \simeq  \left(Rhom_{\Sigma^\infty (\Omega M_+)} (\bs, A^c) \right) \wedge S^n  \xr{1 \wedge t_M}   \\  
 & \left(Rhom_{\Sigma^\infty (\Omega M_+)} (\bs, A^c) \right) \wedge \left( \bs \wedge^L_{\sgm} \bst \right) 
 \xr{\cap} \bst \wedge^L_{\sgm} A^c \simeq THH(A, P). \notag 
\end{align}
 This is the inverse to the equivalence above (\ref{thhap}).

 Now $THH^\bullet(A,A)$ is an $E_2$-ring spectrum. Let $\iota : \bs \to THH^\bullet(A,A)$ be the unit. Alternatively, recall that $THH^\bullet(A,A)$ is the spectrum of $A$-bimodule maps $A \to A$; $\iota$ corresponds to the identity map  $id: A \to A$, a characterization which does not rely on the multiplicative structure on $THH^\bullet(A,A)$. The map $\iota$ allows us to define our $n$-dimensional cotrace map as 
  \begin{equation}\label{sigma}
  \sigma : \Sigma^n\bs \xr{\Sigma^n\iota } \Sigma^nTHH^\bullet (A,A) \simeq THH(A,P).
  \end{equation}
 
Clearly taking cap product
$$
\cap \sigma :  A^! =   Rhom_{\aop}(A, \aop)  \to  \Sigma^{-n} P
$$
defines the equivalence given in (\ref{bimoddual}).  This then proves that $(A, P, \sigma)$ is a twisted, smooth Calabi-Yau ring spectrum.

\med
Let $E$ be any commutative ring spectrum satisfying the hypotheses of the theorem.   An $E_*$-orientation of $M$ is given by a Thom class $\tau : \mtm \to E$  and induces an   an equivalence (\ref{diag})  $\theta_\tau : \Sigma^n \mtm \wedge E \xr{\simeq} M_+ \wedge E$ as in (\ref{diag}).  This can be also be written as an equivalence 
$$
\theta_\tau : (\bst \wedge^L_{\Sigma^\infty (\Omega M_+)} \bs ) \wedge E \xr{\simeq}   (\bs \wedge^L_{\Sigma^\infty (\Omega M_+)} \bs) \wedge E.
$$
Given such an orientation  $\tau : E \to M$ and an orientation $\nu : A = (\Omega M)^{\Omega f} \to E$ we describe a resulting  $E_*$-orientation $(u, \tilde \sigma_E)$ of the $sCY$ structure on $A = (\Omega M)^{\Omega f}$.  

\med
Notice that the $E_*$ orientation $\tau$ of $M$ induces an equivalence   
$$
\theta_{\tau, R} : (\bst \wedge^L_{\Sigma^\infty (\Omega M_+)} R)\wedge E \xr{\simeq}  (\bs \wedge^L_{\Sigma^\infty (\Omega M_+)} R)\wedge E
$$
for any left $\Sigma^\infty (\Omega M_+)$-module $R$.   Now take $R = \aop$ with the $\Sigma^\infty (\Omega M_+)$-action defined by the ring homomorphism $\Sigma^\infty (\Omega M_+) \to \aop$ described above.  We then define 
\begin{equation}\label{udef}
u = \theta_{\tau, \aop} :  P \wedge E =  (\bst \wedge^L_{\Sigma^\infty (\Omega M_+)} \aop)\wedge E \xr{\simeq}  (\bs \wedge^L_{\Sigma^\infty (\Omega M_+)} \aop)\wedge E = A\wedge E.
\end{equation}

\med
We now define the map $\tilde \sigma_E : \bs \wedge E \to THH (A, A)^{hS^1} \wedge E$ needed for the orientation of the $sCY$-structure.

Again, let $\tau : \mtm \to E$ and $\nu : A = (\Omega M)^{\Omega f} \to E$ be orientations of $M$ and $\Omega f$ respectively. Consider the following homotopy commutative diagram.

$$
\begin{CD}
\Sigma^n \bs \wedge E @>\iota >> \Sigma^n THH^\bullet(A, A) \wedge E   @>\simeq >> THH(A, P) \wedge E  \\
&&& & @V\tau V\simeq V  \\
&&@V= VV   THH (A, A) \wedge E  \\
& &     & &  @VV= V    \\
 @A= AA \Sigma^n Rhom_{\sgm}(\bs, A^c) \wedge E     @>\simeq >>      \left(\bs \wedge^L_{\sgm} A^c\right)\wedge E  \\
& & @V \nu V \simeq V     @V \simeq V \nu V \\
& &\Sigma^n Rhom_{\sgm}(\bs, (\sgm)^c) \wedge E     @>\simeq >>      \left(\bs \wedge^L_{\sgm} (\sgm)^c\right)\wedge E \\
&& @AA\iota \wedge 1 A    @AA\iota \wedge 1 A \\
\Sigma^n \bs \wedge E @>\iota >> \Sigma^n Rhom_{\sgm}(\bs, \bs) \wedge E  @>>\simeq >  \left(\bs \wedge^L_{\sgm} \bs\right)\wedge E \\
& & @A= AA     @AA = A \\
& & \Sigma^n M^\vee \wedge E  @>>\simeq > M_+\wedge E
\end{CD}
$$
A few comments about this diagram:
\begin{enumerate}
\item The maps in this diagram that are labelled by a $\tau$ or a $\nu$ are induced by the respective orientations.  The maps
labeled by  $\iota$ are induced by the units of the respective ring spectra. 
\item The reason the left side of this diagram homotopy commutes is because the   vertical maps induced by both $\nu$ and $\iota$  are  all  maps of ring spectra.  As pointed out above this is because the orientation map $\nu : A \to E$ is assumed to be a ring map.
\item The bottom horizontal composition  $\Sigma^n \bs \wedge E \to M_+ \wedge E$ is the $E_*$-fundamental class.  The reason the right side of this diagram homotopy commutes is by the naturality of the Atiyah-Klein equivalences.  
\item The homotopy orbit spectrum $\bs \wedge^L_{\sgm} (\sgm)^c$ is equivalent to $\Sigma^\infty (LM_+)$, and the lower right hand vertical map $M_+ \wedge E \to  \left(\bs \wedge^L_{\sgm} (\sgm)^c\right)\wedge E \simeq \Sigma^\infty (LM_+) \wedge E$
is homotopic to the inclusion of the constant loops, and therefore factors through the homotopy fixed point set of the circle action.
$ M_+ \wedge E  \to \Sigma^\infty (LM_+)^{hS^1}$. 
\item The top horizontal composition is the cotrace map $\sigma \wedge 1 : \Sigma^n \bs  \wedge E \to THH(A, P)\wedge E$.
\end{enumerate}

We therefore define the map $\ \tilde \sigma_E: \Sigma^n E  \to  THH(A, \,  A)^{hS^1} \wedge E$ to be the composition
$$
\begin{CD}
\tilde \sigma_E : \Sigma^n \bs \wedge E @>[M]>>   M_+\wedge E @>>> \Sigma^\infty (LM_+)^{hS^1}  \wedge E  \\
  @>\simeq >>   THH(\sgm, \sgm)^{hS^1} \wedge E   @<\simeq <\nu <  THH(A, A)^{hS^1}\wedge E. \ 
\end{CD}
$$
By comments 3. and 4. above, the composite $ \Sigma^n \bs \wedge E  \xr{\tilde \sigma_E} THH(A, A)^{hS^1} \wedge E \hk THH(A, A) \wedge E$ is   obtained by starting at the lower left of the diagram, going horizontally to the lower right,
and then going vertically until $THH(A, A)\wedge E$.  And by comment 1  above  and the commutativity of the diagram, this means we may conclude that the following diagram homotopy commutes:
$$
\begin{CD}
\Sigma^n \bs \wedge E @>\sigma \wedge 1 >> THH(A, P) \wedge E @>\tau >\simeq > THH(A, A) \wedge E \\
@A= AA   &&  @AAA \\
\Sigma^n \bs \wedge E   &  @>>\tilde \sigma_E > & THH(A, A)^{hS^1}\wedge E.
    \end{CD}
$$

This is what was required to show that $(u, \tilde \sigma_E)$ defines an orientation on the twisted Calabi-Yau structure on $A$. 
 \end{proof}

\med

\medskip

\noindent {\bf Example.} Take $M = SU(m)/SO(m)$, and $f: M \to SU/SO \simeq B^2 O$ the natural map. Then

$$\Omega f: \Omega (SU(m)/SO(m)) \to BO$$
is a loop map, and by Theorem \ref{prop-thom-scy}, $A = \Omega M ^{\Omega f}$ has the structure of a twisted sCY ring spectrum.

Note that $\Omega f: \Omega (SU(m)/SO(m)) \to BO$ induces a map of Thom spectra $A \to MO$.
$\Omega f$ is an equivalence in a range, and therefore induces an equivalence in a range between $A$ and $MO$. $\pi_*(MO)$ is 2-torsion, therefore $\pi_*(A)$ is 2-torsion in a range; equivalently, $A$ is highly connected when localized at an odd prime. This is not the case if $f$ is taken to be nullhomotopic, in which case the $ sCY$ ring spectrum is $\Sigma^\infty (\Omega (SU(m)/SO(m))_+)$.
 
\subsection{The image of $J$ and Lagrangian immersions of spheres} 

In this subsection we study in more detail the twisted smooth Calabi-Yau structure on Thom spectra of virtual bundles over spheres.  These bundles arise naturally from the homotopy perspective from the image of the $J$-homomorphism, and from the perspective of symplectic topology from Lagrangian immersions of odd dimensional spheres into their cotangent bundles. 

As discussed in \cite{ak}, Gromov's $h$-principle   implies that the homotopy group $\pi_n(U)$ classifies Lagrangian immersions of $S^{n}$ into its cotangent bundle, $T^*S^{n}$, which are in the homotopy class of the zero section $S^n \hk T^*(S^n)$.    Assume that $n > 1$ and let  $\alpha : S^n \to U$ represent such a homotopy class. Since $\pi_n(SU) \cong \pi_n(U)$, $\alpha$ lifts to a unique (up to homotopy) map that by abuse of notation we still call $\alpha : S^n \to SU$.  Taking loop spaces we get a map of 
$A_\infty$ group-like monoids,
$$
\Omega \alpha : \Omega S^n \to \Omega SU \simeq BU.
$$
The last equivalence is given by Bott periodicity.  By forgetting the almost complex structure we get an $A_\infty$-map
$$
\Omega \alpha :  \Omega S^n \to BO.
$$
By Theorem \ref{prop-thom-scy}  above, the Thom spectrum $\thomsn$ has the structure of a twisted,  smooth Calabi-Yau ring spectrum.  We begin with the following observation.

\begin{lemma}\label{sorient}  The twisted sCY ring spectrum $\thomsn$ has natural orientation with respect to   stable homotopy theory (that is  the generalized homology theory $\bs_*$ represented by the sphere spectrum $\bs$).      Furthermore this induces an orientation with respect to any generalized homology theory $E_*$ represented by a commutative ring spectrum $E$.
\end{lemma}

\begin{proof} First note that $S^n$ has a canonical stable framing.  That is, it has a canonical $\bs$-orientation.  This induces an orientation with respect to any theory $E_*$.   Furthermore,  by the construction of the twisted sCY structure in the proof of Theorem \ref{prop-thom-scy}, the twisting bimodule of this structure is
$$
P = \bst \wedge^L_{\Sigma^\infty (\Omega S^n_+)} A \wedge A^{op}    
$$
where $A = \thomsn$.   Now the $\bs$-framing of $S^n$ defines an    equivalence of bimodules, $\bs \simeq \bst$.  Thus
$$
P \simeq \bs\wedge^L_{\Sigma^\infty (\Omega S^n_+)} A \wedge A^{op} 
$$
but by (\ref{thm5.1}) this last spectrum is equivalent to $A$ as $A$-bimodules.  Thus we have an equivalence
$$
u : P \wedge \bs \xr{\simeq} A\wedge \bs.
$$

Using this identification,  the cotrace element can be viewed as a class $\sigma \in THH(A, A)$.  To complete the construction of the $\bs_*$-orientation we must show
that $\sigma$ lifts to an element in the homotopy fixed points, $THH(A, A)^{hS^1}$. 
   
   By the main result of \cite{BCS}, the topological Hochschild homology of  $A = \thomsn$ is equivalent as a $\Sigma^\infty(S^1_+)$-module to the Thom spectrum of a virtual bundle over the free loop space $LS^n$:
   
   \med
   \begin{proposition}\label{BCS}\cite{BCS}
   $$ THH (\thomsn) \simeq L(S^n)^{\ell (\alpha)}.$$
where $\ell (\alpha )$ is the virtual bundle classified by the map
   $$
   \ell (\alpha )  :  L(S^n)  \xr{L\alpha} LSU \simeq SU \times \Omega SU \xr{project} \Omega SU \simeq BU \to BO.
   $$
      \end{proposition}
   In this composition, the equivalence $LSU \simeq SU \times \Omega SU$ is given by the trivialization of the fibration of infinite loop spaces $\Omega SU \to LSU \to SU$ defined by the canonical section $SU \to LSU$ given by the inclusion of constant loops, and the infinite loop structure of $LSU$.

   \med
   \noindent {\bf Remark. }   In \cite {BCS} the map from $LSU$ to $BO$ was described by a composition
   $$
   LSU \to L(SU/SO) \simeq  SU/SO \times \Omega (SU/SO) \xr{\eta \times 1} \Omega (SU/SO) \times \Omega (SU/O) \xr{multiply} \Omega (SU/SO) \simeq BO
   $$
   where $\eta : SU/SO \to \Omega (SU/SO) \simeq BO$ was induced by the Hopf map $S^3 \to S^2$.  However the map $\eta$  becomes trivial when composed with the projection
   $SU \to SU/SO$ which allows the description of $\ell (\alpha )$ given in the proposition.

\med
Notice that the restriction to the constant loops,
$$
S^n \xr{\iota}  LS^n \xr{\ell (\alpha)} BO
$$
is the constant map.  That is, this virtual bundle is trivialized when restricted to the constant loops.  But since constant loops are $S^1$-fixed points of $LS^n$, the inclusion naturally lifts to the homotopy fixed points,
$$
\Sigma^\infty (S^n_+) \xr{\iota } (L(S^n)^{\ell (\alpha)})^{hS^1}.
$$
Composing with the equivalence given by Proposition \ref{BCS}, this defines a map $\tilde \sigma : \Sigma^\infty (S^n) \to THH(A, A)^{hS^1}$ that lifts the cotrace element $\sigma \in THH(A, A)$.   

 This completes the construction of the $\bs_*$-orientation of the sCY structure on $A = \thomsn$.   Given any other generalized homology theory $E_*$ represented by a commutative ring spectrum $E$, the $\bs_*$-orientation of $\thomsn$ induces an $E_*$ orientation by use of the unit $\bs \to E$.   This completes the proof of Lemma \ref{sorient}.
\end{proof}

The following recasts the results of \cite{ak} to show that topological Hochschild homology can be used as an obstruction to being able to deform a Lagrangian immersion of a sphere to a Lagrangian embedding.

\med
\begin{theorem}\label{lagTHH} Let $\alpha : S^n \to U$ represent a Lagrangian immersion $\phi_\alpha : S^n \to T^*S^n$   in the homotopy class of the zero section.  Consider the associated twisted smooth Calabi-Yau ring spectrum $(\Omega S^n)^{\Omega (-\alpha)}$. (Here $-\alpha : S^n \to U$ is a map that represents the inverse of $\alpha$ in $\pi_nU$.) Then if $\phi_\alpha$ is Lagrangian isotopic to a Lagrangian embedding then there is an equivalence of topological Hochschild homology spectra,
$$
THH ((\Omega S^n)^{\Omega (-\alpha)}) \simeq THH(\Sigma^\infty (\Omega S^n_+)).
$$
\end{theorem}

\begin{proof}   Let $Q$ and $N$ be smooth, closed manifolds of the same dimension. Given an exact Lagrangian embedding $j : Q \to T^*N$,  Kragh in \cite{kr} defined  
a virtual Maslov bundle $\nu$ on $L_0Q$.  Here $L_0$ denotes path component of the free loop space that contains constant loops.  The construction, which uses notation that is different than ours, is described in section 2 of \cite{ak}.   A map of spectra 
\begin{equation}\label{equiv}
\psi : \Sigma^\infty(L_0N_+)  \to L_0Q^{TN -TL\oplus \nu}
\end{equation}
was constructed and studied. One of the main results of \cite{ak} is that the map $\psi$ is a homotopy equivalence of spectra.  
The Maslov bundle $\nu$ was defined as follows.      (See section 2 of \cite{ak}.)   The Lagrangian embedding $Q\to T^*N$ defines a map $\tau :  Q \to U/O$.  Then $-\nu$ was defined to be the restriction to $L_0Q$ of the map
$$
LQ \xr{L\tau}  L(U/O) \xr{\simeq} U/O \times \Omega U/O \xr{project}\Omega U/O \simeq \bz \times BO.
$$
In considering  a Lagrangian embedding or immersion  $\phi_\alpha : S^n \to T^*S^n$ represented by $\alpha : S^n \to U$, then by Proposition \ref{BCS} the Maslov bundle $\nu$ is just $-\ell (\alpha) : LS^n \to BU \to BO.$    Thus we may conclude from (\ref{equiv}) that  if the Lagrangian immersion $\phi_\alpha$ is Lagrangian isotopic to a Lagrangian embedding, then the spectra
$$
\Sigma^\infty(LS^n_+)   \quad \text{and} \quad (LS^n)^{-\ell (\alpha)} = (LS^n)^{\ell(-\alpha)}
$$
are equivalent.


Now the spectrum $\Sigma^\infty(LS^n_+) $ is equivalent to the topological Hochschild homology $THH(\Sigma^\infty(\Omega S^n_+))$.  Whereas by Proposition \ref{BCS}, the spectrum $(LS^n)^{\ell (-\alpha)}$ is equivalent to the topological Hochschild homology $THH((\Omega S^n)^{\Omega(-\alpha)} )$.  The statement of the theorem now follows.
\end{proof}

For $k > 1$, let $\alpha_k : S^{2k+1} \to U$ be a generator of $\pi_{2k+1}(U)$ which by Bott periodicity is isomorphic to the integers.   In     \cite{ak} it was proved  that for $2k+1$  congruent to $1$, $3$, or $5$  $ modulo \, 8$, the Lagrangian immersion  $\phi_k : S^{2k+1} \to T^*(S^{2k+1})$ represented by $\alpha_k$ is not Lagrangian isotopic to a Lagrangian embedding.  We now see that this is detected by the fact the twisted smooth Calabi-Yau ring spectra $\Sigma^\infty (\Omega S^{2k+1}_+)$ and $\thomk$ have different topological Hochschild homologies.  For this we again use the fact that 
$$
THH(\Sigma^\infty (\Omega S^{2k+1}_+)) \simeq \Sigma^\infty (LS^{2k+1}_+)    \quad \text{and} \quad  THH(\thomk) \simeq (LS^{2k+1})^{\ell (-\alpha_k)}.
$$

In \cite{susp}, it was shown that the free loop space of the suspension of a connected, based space $X$ has a model
$$
L\Sigma X \simeq \bigcup_n S^1 \times_{\bz/n} X^n  / \sim
$$
where the equivalence relation $\sim $ is a basepoint identification.   This space is equivalent to configuration space of points in $S^1$ with labels in $X$:
\begin{equation}\label{cs1}
C(S^1, X) = \bigcup_n F(S^1, n) \times_{\Sigma_n} X^n / \sim \quad  \simeq L\Sigma X.
\end{equation}
where $ F(S^1, n)$ is the configuration space of $n$ ordered distinct points in $S^1$. It was also shown in \cite{susp} that the suspension spectrum of these filtered models stably split:
$$
\Sigma^\infty (L\Sigma X_+) \simeq \bs \vee \Sigma^\infty\left( \bigvee_n S^1_+ \wedge_{\bz/n} X^{(n)}\right)
$$
where  $X^{(n)}$ denotes the $n$-fold smash product. 

Thus we have 
\begin{equation}\label{split}
THH(\Sigma^\infty (\Omega S^{2k+1}_+)) \simeq \Sigma^\infty (LS^{2k+1}_+)  \simeq \bs \vee \Sigma^\infty \left(\bigvee_n S^1_+ \wedge_{\bz/n} S^{2kn}\right).
\end{equation}

On the other hand, applying   a result of Lewis in \cite{LMS} to (\ref{cs1}) says that the   Thom spectrum $ (LS^{2k+1})^{\ell (-\alpha_k)}$ has the homotopy type of
the spectrum $C(S^1, (S^{2k})^{b\alpha_k})$ where here
$
C(S^1, -)$ is viewed as a functor from the category of spectra $E$ with unit $\bs \to E$, to spectra.   $(S^{2k})^{b\alpha_k}$ is the Thom spectrum of the map
$$
b\alpha_k : S^{2k} \xr{\iota} \Omega \Sigma S^{2k} = \Omega S^{2k+1} \xr{\Omega (-\alpha_k)} \Omega U \simeq \bz \times BU \to \bz \times BO.
$$
 Here $\iota : S^{2k} \to \Omega \Sigma S^{2k}$ is the adjoint of the identity on $\Sigma S^{2k}$.  $\iota$ generates $\pi_{2k}\Omega \Sigma S^{2k}  \cong \bz$.  Note  also that by the connectivity of $S^{2k}$  the image of  $b\alpha_k$ lies in $\{0\} \times BO$. 
 
 Think of $\Sigma_{m-1}$ as the subgroup of $\Sigma_m$ consisting of those permutations of $m$-letters that leave the last letter fixed.  Then as in \cite{LMS}, the spectrum $ C(S^1, (S^{2k})^{b\alpha_k}) \simeq  (LS^{2k+1})^{\ell (-\alpha_k)} $ is defined to be the homotopy coequalizer of the maps
 $$
 \beta, \gamma : \bigvee_m F(S^1, m)_+ \wedge_{\Sigma_{m-1}} \left((S^{2k})^{b\alpha_k})\right)^{(m-1)} \la   \bigvee_n F(S^1, n)_+ \wedge_{\Sigma_{n}} \left((S^{2k})^{b\alpha_k})\right)^{(n)} 
 $$
 where for each $m$,
  $$\beta : F(S^1, m)_+ \wedge_{\Sigma_{m-1}} \left((S^{2k})^{b\alpha_k})\right)^{(m-1)}  \to F(S^1, m-1)_+ \wedge_{\Sigma_{m-1}} \left((S^{2k})^{b\alpha_k})\right)^{(m-1)} $$
 is defined by deleting the last coordinate of the configuration in $F(S^1, m)$.   Whereas
  $$
  \gamma :    F(S^1, m)_+ \wedge_{\Sigma_{m-1}} \left((S^{2k})^{b\alpha_k})\right)^{(m-1)}  \to F(S^1, m)_+ \wedge_{\Sigma_{m}} \left((S^{2k})^{b\alpha_k})\right)^{(m)} 
  $$
 is given by smashing with the unit, $ \left((S^{2k})^{b\alpha_k})\right)^{(m-1)}\wedge \bs \xr{1\wedge u} \left((S^{2k})^{b\alpha_k})\right)^{(m)}$.  
 
 \med
 We make two observations about this construction.   
 \begin{enumerate}\label{observe}
 \item Notice that the Thom spectrum $(S^{2k})^{b\alpha_k}$ is equivalent to the $CW$ spectrum
 $$
 (S^{2k})^{b\alpha_k} \simeq \bs \cup_{\tilde \alpha_k} D^{2k}
 $$
 where the attaching map   $\tilde \alpha_k : \Sigma^\infty S^{2k-1} \to \bs$ is  defined as follows. 
 
 Consider the composition $S^{2k} \xr{b\alpha_k} BO \xr{BJ} BGL_1(\bs)$
 where $BJ$ is the delooping of the $J$-homomorphism $J : O \to GL_1(\bs)$.  Applying the loop space defines a map $S^{2k-1} \to GL_1(\bs)$.  Since $S^{2k-1}$ is connected, its image lies in a single component of $GL_1(\bs)$ which is equivalent to the component of the basepoint in $QS^0$.  The adjoint of this map is the definition of the map
 $\tilde \alpha_k : \Sigma^\infty (S^{2k-1}) \to \bs$.  Notice, that by definition it is in the image of the $J$ homomorphism,  $J : \pi_{2k-1}O \to \pi_{2k-1}(\bs)$. 
 
 \item   $ C(S^1, (S^{2k})^{b\alpha_k}) \simeq C(S^1, \bs \cup_{\tilde \alpha_k} D^{2k})$ is a naturally filtered spectrum, where the filtration is by the cardinality of the configuration of points in $S^1$.   Let $\cf_m( C(S^1, \bs \cup_{\tilde \alpha_k} D^{2k})$ be the $m^{th}$ filtration. We observe that filtration $0$ part $\cf_0$  is the unit $\bs \hk C(S^1, \bs \cup_{\tilde \alpha_k} D^{2k})$, while filtration 1 part is  $ \cf_1( C(S^1, \bs \cup_{\tilde \alpha_k} D^{2k}) = S^1_+ \wedge \bs \cup_{\tilde \alpha_k} D^{2k}$.
  \end{enumerate}
  
  \med
  By the inclusion of the filtration 0 part, we have a cofibration sequence of spectra,
   \begin{equation}\label{cofib}
   \bs \to C(S^1, \bs \cup_{\tilde \alpha_k} D^{2k}) \simeq THH(\thomk) \to C(S^1, \bs \cup_{\tilde \alpha_k} D^{2k}) /\bs.
   \end{equation}.
   
   In the case where $\alpha_k$ is replace by the trivial homotopy class, this cofibration sequence becomes
   $$
   \bs \to C(S^1, \Sigma^\infty (S^{2k}_+)) \simeq THH( \Sigma^\infty (\Omega S^{2k+1}_+)) \to C(S^1,   \Sigma^\infty (S^{2k}_+))/\bs.
   $$
By the splitting given by  (\ref{split})  above,   this cofibration splits, and  $$C(S^1,   \Sigma^\infty (S^{2k}_+))/ \bs \simeq  \Sigma^\infty \left(\bigvee_n S^1_+ \wedge_{\bz/n} S^{2kn}\right).$$

So to prove that $THH(\thomk)$ is \sl not \rm equivalent to $  THH( \Sigma^\infty (\Omega S^{2k+1}_+))$,  it suffices to show that the cofibration sequence (\ref{cofib}) \sl does not \rm split.    But to do this, it suffices to check that the cofibration sequence one gets by restricting to filtration 1 of  $C(S^1,  C(S^1, \bs \cup_{\tilde \alpha_k} D^{2k}))$ 
does not split.  That is, by the second observation above, it suffices to check that the cofibration sequence
$$
\bs \to S^1_+ \wedge \left( \bs \cup_{\tilde \alpha_k} D^{2k}\right) \to \left(S^1_+\wedge S^{2k}\right) \vee S^1_+
$$ 
does not split.  But this is true so long as we know that the attaching map $\tilde \alpha_k : \Sigma^\infty (S^{2k-1}) \to \bs$ is nontrivial. 
But this class is in the image of $J$, constructed out of $\alpha_k : S^{2k+1} \to U$ as described above.
By standard calculations of Quillen and Adams, as described in \cite{ak}, for $ 2k+1$  congruent to $1$, $3$, or $5$  $ modulo \, 8$,  this class is nontrivial.
Hence the cofibration sequence (\ref{cofib}) does not split which means that $THH(\thomk)$ and $THH( \Sigma^\infty (\Omega S^{2k+1}_+))$ are not homotopy equivalent spectra.  By Theorem \ref{lagTHH} this implies that the Lagrangian immersion $\phi_k$ is \sl not \rm Lagrangian isotopic to a Lagrangian embedding.

\section{A topological Hochschild (co)homology perspective }

In this section we give a topological Hochschild homology and cohomology interpretation of the Calabi-Yau structures and the dualities between the manifold string topology ring spectrum $\cs^\bullet_{P}(M)$ and the Lie group string topology coalgebra spectrum $\cs^P_\bullet (M)$.

\med
We continue to consider a principal bundle $G \to P \xr{p} M$ where $M$ is a closed manifold of dimension $n$ and $G$ is a compact Lie group of dimension $d$.  A choice of connection on the bundle $P$ defines a \sl holonomy \rm map
$$
h_P : \Omega M \to G.
$$
This is a map of group-like $A_\infty$ spaces, and the induced map of classifying spaces, $Bh_P : M \simeq B(\Omega M) \to BG$ classifies the bundle $P$.   

We then have an induced map of ring spectra and differential graded algebras that by abuse of notation we still denote by $h_P$:
$$
\Sigma^\infty (\Omega M_+) \xr{h_P} \Sigma^\infty (G_+)    \quad \text{and} \quad C_*(\Omega M) \xr{h_P} C_*(G)
$$
These holonomy maps therefore define bimodule structures of $\Sigma^\infty (G_+)$ over $\Sigma^\infty (\Omega M_+)$ and of $C_*(G)$ over $C_*(\Omega M)$.     We can therefore study the (topological) homology of these algebras with coefficients in these bimodules.  In what follows we suppress the map $h_P$ from the notation regarding these bimodules.  This abuse is somewhat justified because given any two choices of holonomy maps, the induced module structures will be equivalent.   We also note that a choice of holonomy defines an inherited dual bimodule structure on  the Spanier-Whitehead dual $G^\vee$ over   $\Sigma^\infty (\Omega M_+)$, and similarly the cochains $C^*(G)$ inherit the dual bimodule structure over $C_*(\Omega M)$.  One of the main results of this section is the following. 

\med
\begin{theorem}\label{hochschild}  We have the following equivalences involving topological Hochschild homology $THH_\bullet$ and topological Hochschild cohomology $THH^\bullet$.
\begin{enumerate} 
\item $THH_\bullet(\Sigma^\infty (\Omega M_+), \Sigma^\infty (G_+)) \simeq \Sigma^\infty(\pad_+)$ \notag \\
\item $ THH^\bullet (\Sigma^\infty (\Omega M_+), \Sigma^\infty (G_+)) \simeq (\pad)^{-TM}  \simeq \cs_P^\bullet (M).   $ \quad \text{This equivalence is one of ring spectra.} \\
\item $THH_\bullet(\Sigma^\infty (\Omega M_+), G^\vee) \simeq (\pad)^{-\tv} \simeq \cs_\bullet^P(M).$  \quad \text{This equivalence is one of coalgebra spectra.} 
\end{enumerate}
\end{theorem}

\begin{proof}  Given any   homomorphism $\phi : H \to G$ of topological groups,  then one has that $$THH_\bullet(\Sigma^\infty (H_+), \Sigma^\infty (G_+)) \simeq \Sigma^\infty (EH \times_H G^{Ad})$$
where $G^{Ad}$ represents the adjoint (conjugation) action of $H$ on $G$:   $h\cdot g = \phi(h)g\phi(h)^{-1}$.  This is because $THH_*(\Sigma^\infty (H_+), \Sigma^\infty (G_+)) $ is equivalent to the suspension spectrum of the cyclic bar construction $N^{cy}(H,G)$ which Waldhausen showed is equivalent to the homotopy orbit space of $H$ acting on $G$ via the conjugation action \cite{waldhausen}.  In our case, we may think of $H$ as the based loop space $\Omega M$ by taking $H$ to be a topological group of the same $A_\infty$-homotopy type.  (As we did earlier, by abuse of notation we still call this group $\Omega M$.) Then this  observation says that
$$
THH_\bullet(\Sigma^\infty (\Omega M_+), \Sigma^\infty (G_+)) \simeq  \Sigma^\infty(E\Omega M \times_{\Omega M} G^{Ad}_+) \simeq \Sigma^\infty (\pad_+).
$$
This proves part (1) of the theorem.

For part (2), we use the similarly well-known fact that the topological Hochschild cohomology of the suspension spectrum of a group can be described as a homotopy fixed point spectrum.  That is, like above, let $\phi : H \to G$ be a   homomorphism of topological groups, then

\begin{equation}
THH^\bullet (\Sigma^\infty (H_+),  \Sigma^\infty(G_+))  \simeq 
\Sigma^\infty(G_+)^{h\Sigma^\infty (H_+)}
\end{equation}
where $\Sigma^\infty (H_+)$ acts on $\Sigma^\infty (G_+)$ via the conjugation action.   Like above we refer to this as the adjoint action and we write it as $\Sigma^\infty (G_+)^{Ad}$. (See \cite{westerland} or   section 4 of \cite{malm}.)   

  Now this homotopy fixed point spectrum is  defined to be
$$
\Sigma^\infty(G_+)^{h\Sigma^\infty (H_+)}   =  Rhom_{\Sigma^\infty (H_+)}(\bs,    \Sigma^\infty (G_+)^{Ad}).
$$
So in our case we have that
$$
 THH^\bullet (\Sigma^\infty (\Omega M_+), \Sigma^\infty (G_+))   \simeq Rhom_{\Sigma^\infty (\Omega M_+)}(\bs,    \Sigma^\infty (G_+)^{Ad}).
 $$
 Now notice that since the homotopy orbit spectrum of  $\Sigma^\infty (\Omega M_+)$ acting on $ \Sigma^\infty (G_+)$  via the adjoint action is $\Sigma^\infty (\pad)$,   this spectrum of $\Sigma^\infty(\Omega M_+)$-equivariant morphisms is equivalent to the spectrum of sections of  the parameterized spectrum $\Sigma^\infty (G_+) \to \Sigma^\infty_M ((\pad)_+) \to M$:
$$
 Rhom_{\Sigma^\infty (\Omega M_+)}(\bs,    \Sigma^\infty (G_+)^{Ad}) \simeq  \Gamma_M(\Sigma^\infty_M ((\pad)_+)).
 $$
 But this spectrum of sections is, by definition, the manifold string topology spectrum, $\cs^\bullet_P(M)$. Furthermore it is clear that the ring spectrum structures coincide under this equivalence.  Furthermore,
  by the Atiyah-Poincar\'e duality theorem proved by Klein \cite{klein}, \cite{CohenKlein} we have that
 $$
  \Gamma_M(\Sigma^\infty_M ((\pad)_+)) \simeq (\pad)^{-TM}
  $$
  as ring spectra.
  Putting these together says that 
  $$
  THH^\bullet (\Sigma^\infty (\Omega M_+), \Sigma^\infty (G_+))  \simeq  \cs^\bullet_P(M) \simeq (\pad)^{-TM} 
  $$
  as ring spectra.  This is the statement of part (2) of the theorem.
  
  We now consider part (3) of the theorem.
  The Spanier-Whitehead dual of the simplicial spectrum $THH_\bullet (\Sigma^\infty (\Omega M_+), G^\vee)$ can, because of the compactness assumption of $G$, be described  as the totalization of the  cosimplicial spectrum given by taking the Spanier-Whitehead dual levelwise.  This cosimplicial spectrum has as its spectrum of $k$-simplices, $Rhom_\bs ( \Sigma^\infty (\Omega M_+)^{(k)}, \Sigma^\infty (G_+))$.  The coface maps and the codegeneracies are the duals of the face and degeneracy maps in the simplicial spectrum $THH_\bullet (\Sigma^\infty (\Omega M_+), G^\vee)$.  But this cosimplicial spectrum is exactly the cosimplicial spectrum defining the topological Hochschild cohomology spectrum,
  $THH^\bullet (\Sigma^\infty (\Omega M_+), \Sigma^\infty (G_+))$.  That is, we have observed that
  $$
  THH_\bullet (\Sigma^\infty (\Omega M_+), G^\vee)^\vee = THH^\bullet (\Sigma^\infty (\Omega M_+), \Sigma^\infty (G_+)).
  $$
This is a ring spectrum, so its Spanier-Whitehead dual, $  THH_\bullet (\Sigma^\infty (\Omega M_+), G^\vee)$ inherits the structure of a coalgebra spectrum.   Furthermore, we know from part (2) that 
  $$
  THH_\bullet (\Sigma^\infty (\Omega M_+), G^\vee)^\vee = THH^\bullet (\Sigma^\infty (\Omega M_+), \Sigma^\infty (G_+)) \simeq \cs^\bullet_P(M) \simeq (\pad)^{-TM}
  $$
  as ring spectra.  Thus applying Spanier-Whitehead duality, Theorem \ref{SWdual}, and Theorem \ref{main} we have
  $$
  THH_\bullet (\Sigma^\infty (\Omega M_+), G^\vee) \simeq \cs^P_\bullet (M) \simeq (\pad)^{-\tv}
  $$
  as coalgebra spectra. 
 
\medskip

Alternatively, as in the proof of part (1) of the theorem, we have an equivalence

$$THH_\bullet(\Sigma^\infty(\Omega M _+), G^\vee) \simeq \bs \wedge^L _{\Sigma^\infty(\Omega M_+)} (G^\vee)^{Ad}$$

Now, $\cs_\bullet^P(M)$ is the homology spectrum of the spectrum over $M$ whose fiber is $G^\vee$, on which $\Omega M$ acts via (the dual of the) conjugation action. Therefore, we see that

$$\cs_\bullet^P(M) \simeq \bs \wedge^L _{\Sigma^\infty(\Omega M_+)} (G^\vee)^{Ad} \simeq THH_\bullet(\Sigma^\infty(\Omega M _+), G^\vee)$$
  
For all these spectra, the coproduct comes from dualizing the multiplication map $G \times G \to G$; see below for an explicit description of the coproduct on $\bs \wedge^L _{\Sigma^\infty(\Omega M_+)} (G^\vee)^{Ad}$. 
 \end{proof}

\med 
 We end by observing how the twisted compact Calabi-Yau structure on $\cs^\bullet_P(M)\simeq (\pad)^{-TM}$ can be understood from this Hochschild perspective.  
 
 \med
 The twisting bimodule in the twisted $cCY$ structure on $R = \cs^\bullet_P(M)$ is $Q = \Sigma^{d-n}(\cs_\bullet^P(M)) \simeq \Sigma^{d-n}(\pad)^{-\tv}$. 
 We first observe that the duality pairing (\ref{pairing})  in the dimension $n-d$  twisted compact Calabi-Yau structure
 $$
 \langle , \rangle : Q \wedge R \to  \Sigma^{d-n}\bs
 $$
 can be described in terms of Hochschild theory as follows.   As described above we have natural equivalences 
 \begin{align}
 R = THH^\bullet (\Sigma^\infty (\Omega M_+), \Sigma^\infty (G_+)) &\simeq Rhom_{\lom} (\bs, \Sigma^\infty (G_+)^{Ad}), \quad \text{and} \notag \\
 Q = \Sigma^{d-n}THH_\bullet (\lom, G^\vee) &\simeq \Sigma^{d-n} \bs \wedge_{\lom} (G^\vee)^{Ad} \notag
 \end{align}
 
 We therefore have a cap product
 \begin{align}
 \cap : Q \wedge R = \left(\Sigma^{d-n} \bs \wedge_{\lom} (G^\vee)^{Ad}\right)  &\wedge \left(Rhom_{\lom} (\bs, \Sigma^\infty (G_+)^{Ad}\right)   \\ &\la  \Sigma^{d-n}(\omg)^{Ad} \wedge_{\lom}(G^\vee)^{Ad}  \notag
  \end{align}
  The evaluation map $ev: \omg \wedge G^\vee \to \bs$ is $\lom $ - invariant with respect to conjugation, and so defines a map
  $$
  ev : \Sigma^{d-n}(\omg)^{Ad} \wedge_{\lom}(G^\vee)^{Ad} \to \Sigma^{d-n} \bs.
  $$
  Composing these defines the duality pairing
\begin{align}
  \langle , \rangle : Q \wedge R &\to  \Sigma^{d-n}\bs   \notag \\
  \left(\Sigma^{d-n}THH_\bullet (\lom, G^\vee)\right)&\wedge \left(THH^\bullet (\Sigma^\infty (\Omega M_+), \Sigma^\infty (G_+))\right)  \xr{\cap} \notag \\
   \Sigma^{d-n}\left((\omg)^{Ad} \wedge_{\lom}(G^\vee)^{Ad}\right) &\xr{ev} \Sigma^{d-n} \bs.  \notag
  \end{align}

  \med
 We end by observing how the   bimodule structure  of $Q$ over $R$ can be understood at the topological Hochschild (co)homology level.   
 $Q = \Sigma^{d-n}\cs^P_\bullet (M) \simeq \Sigma^{d-n} THH_\bullet (\lom, G^\vee)$. Now $\cs^P_\bullet (M) \simeq THH_\bullet (\lom, G^\vee) \simeq \bs \wedge_{\lom}(G^\vee)^{Ad}$ is a coalgebra spectrum, and its coproduct $\psi$ can be seen on the $THH$-level as follows:
 
$$\xymatrix{
\bs \wedge_{\lom}(G^\vee)^{Ad} \ar[r]^-{1 \wedge \mu^\vee} & \bs \wedge_{\lom}((G \times G)^\vee)^{Ad} & \ar[l]^-{\simeq}  \bs \wedge_{\lom}(G^\vee \wedge G^\vee)^{Ad}  \ar[d]^-{\Delta} \\
& &  (\bs \wedge_{\lom}(G^\vee)^{Ad}) \wedge (\bs \wedge_{\lom}(G^\vee)^{Ad})
}$$

This needs some explanation.   $\mu : G \times G \to G$ is the multiplication map.  $\mu^\vee : G^\vee \to G^\vee \wedge G^\vee$ is its Spanier - Whitehead dual.  It is equivariant with respect to the adjoint action of $\lom$ since $\mu$ is equivariant with respect to the adjoint action.  
The map $\Delta : \bs \wedge_{\lom}(G^\vee \wedge G^\vee)^{Ad} \to  \left(\bs \wedge_{\lom}(G^\vee)^{Ad}\right)   \wedge \left(\bs \wedge_{\lom}(G^\vee)^{Ad}\right)$
is the map induced by thinking of $\Omega M$ as a subgroup of $\Omega M \times \Omega M$ as the diagonal subgroup.

The action map $Q \wedge R \to R$ is then homotopic to the composition
\begin{align}
&Q\wedge R  \simeq  \left( \Sigma^{d-n} THH_\bullet (\lom, G^\vee)\right)  \wedge THH^\bullet (\Sigma^\infty (\Omega M_+), \Sigma^\infty (G_+)) \xr{\psi \wedge 1} \notag \\
&\Sigma^{d-n} THH_\bullet (\lom, G^\vee) \wedge   THH_\bullet (\lom, G^\vee)  \wedge THH^\bullet (\Sigma^\infty (\Omega M_+), \Sigma^\infty (G_+))   \xr{1 \wedge \langle , \rangle } \notag \\
 &\Sigma^{d-n} THH_\bullet(\lom, G^\vee)\wedge \bs = Q. \notag
  \end{align}
  
  The left module structure is homotopic to the analogous composition $R \wedge Q \to Q$.


\begin{thebibliography}{99}{ 

\bibitem{ak} M. Abouzaid and T. Kragh, \emph{On the immersion classes of nearby Lagrangians}, preprint arXiv:1305.6810v3 (2015)
\bibitem{atiyahdual}M.F. Atiyah, \emph{Thom complexes },
 Proc. London Math. Soc. \textbf{ (3) }, no. 11 (1961),  291--310.



   \bibitem{atiyahbott}M. Atiyah, and R. Bott \emph{The Yang-Mills equations over Riemann surfaces}, Phil. Trans. R. Soc. Lond. A  \bf 308, \rm 523-615 (1982)  
 
\bibitem{ABG} M. Ando, A. Blumberg, and D. Gepner \emph{Parametrized spectra, multiplicative Thom spectra, and the twisted Umkehr map}, preprint: https://arxiv.org/pdf/1112.2203.pdf (2011) 
 
 

\bibitem{BGNX} K. Behrend, G. Ginot, B. Noohi, and P.Xu,    \emph{String topology for stacks}, Ast\'erisque {\bf 343} (2012), xiv +169, 

\bibitem{BCS} A. Blumberg, R.L Cohen, and C. Schlichtkrull \emph{Topological Hochschild homology of Thom spectra and the free loop space}, Geometry \& Topology \bf vol 14 no.2 \rm{2010}, 1165-1242.
 
 \bibitem{BCG} A. Blumberg, R.L. Cohen, and S. Ganatra \emph{in progress}
 \bibitem{chassullivan}  M. Chas   and D. Sullivan, \emph{String Topology}.
  preprint: math.GT/9911159.
  
  \bibitem{susp} R.L. Cohen, \emph{ A model for the free loop space of a suspension},  Algebraic topology (Seattle, Wash., 1985),  Lecture Notes in Math., {\bf 1286}, Springer, Berlin, (1987) 193Ð207,
  
\bibitem{chataurmenichi} D. Chataur and L. Menichi, \emph{String topology for classifying spaces}, J. Reine Angew. Math. {\bf 669}, (2012) 1-45.

\bibitem{cohenatiyah} R. L. Cohen,  \emph{Multiplicative properties of Atiyah duality},   Homology, Homotopy, and its Applications, \bf vol 6 no. 1 \rm (2004), 269-281.

\bibitem{cg} R.L. Cohen and S. Ganatra, \emph{ Calabi-Yau categories, the string topology of a manifold, and the Floer theory of its cotangent bundle},   preliminary draft available at:          math.stanford.edu/$\sim$ ganatra/materials/scy-floer-string\_draft.pdf  

  \bibitem{cohenjones} R.L. Cohen and J.D.S. Jones, \emph{ A homotopy theoretic realization 
of string topology},  
  Math. Annalen, \bf  vol. 324,  \rm 773-798 (2002).  preprint: math.GT/0107187 

\bibitem{cjgauge} R.L. Cohen and J.D.S. Jones,  \emph{Gauge theory and string topology}, Mexican Bulletin of Mathematicss, to appear, 2017,  preprint: http://arxiv.org/abs/1304.0613   



\bibitem{CohenKlein}  R.L. Cohen and J.R. Klein, \emph{Umkehr maps},    Homology, Homotopy, and Applications,  \bf vol. 11 \rm (1),  (2009), 17-33.  preprint:
 arXiv:0711.0540 

 \bibitem{costello}  K. Costello, 
 \emph{ Topological conformal field theories and Calabi-Yau categories}, Adv. Math. \bf vol. 210 \rm  (2007), no. 1, 165ÃÂ214.  
 
 \bibitem{DGI} W.G. Dwyer, J.P.C. Greenlees, and S. Iyengar, \emph{Duality in algebra and topology}, Adv. Math. \bf vol. 200, no.2 \rm, (2006), 357-402.
 
 \bibitem{MM}  M.A. Mandell and J.P. May, \textbf{ Equivariant orthogonal spectra and S-modules} , Memoirs Amer. Math. Soc. 755 (2002).
 
 



\bibitem{gruhersalvatore} K. Gruher and P. Salvatore,  \emph{Generalized string topology operations} Proc. Lond. Math. Soc. (3) \bf 96 \rm    (2008),   78-106. 

\bibitem{gruher} K. Gruher, \emph{A duality between string topology and the fusion product in equivariant $K$-theory}, Math. Res. Lett. {\bf 14}, (2007), no. 2, 303-313.

 

 
\bibitem{klang} I. Klang,  \emph{The factorization theory of Thom spectra and twisted non-abelian Poincar\'e duality}, preprint arXiv:1606.03805 (2016)

\bibitem{klein}  J.R. Klein, \emph{The dualizing spectrum of a topological group} \newblock {\it Math. Annalen\rm } {\bf 319}, 421--456 (2001)

\bibitem{kv} M. Kontsevich and Y. Vlassopoulos, \emph{Calabi-Yau infinity algebras and negative cyclic homology}, preprint, 2013.   

\bibitem{ks} M. Kontsevich and Y. Soibelman, \emph{Notes on $A_\infty$-algebras, $A_\infty$-categories, and non-commutative geometry I},  preprint:  arxiv:mathRA/0606241 (2006). 

\bibitem{kr} T. Kragh, \emph{The Viterbo transfer as a map of spectra}, preprint: arXiv:0712.2533  (2007).



\bibitem{LMS} L.G. Lewis,   J.P  May,  and M. Steinberger,   {\bf  Equivariant stable homotopy theory} with contributions by J. E. McClure. Lecture Notes in Mathematics, {\bf 1213},  Springer-Verlag, Berlin, (1986).




\bibitem{lind} J.A. Lind, \emph{Bundles of spectra and Algebraic $K$-theory}, preprint arXiv:1304567


\bibitem{LUX} E. Lupercio, B. Uribe, M.A. Xicotencatl, \emph{Orbifold string topology}, Geom. Topol. {\bf 12} (4), (2008), 2203-2247.

\bibitem{lurie} J. Lurie, \emph{ On the Classification of Topological Field Theories},  preprint:  
  arXiv:0905.0465  (2009)
  
  
  \bibitem{malm} E.J. Malm, \emph{String topology and the based loop space}, preprint: arXiv:1103.6198 (2011)
  
\bibitem{MS} J. P. May and 
 J.  Sigurdsson,  {Parametrized homotopy theory}.
\newblock Mathematical Surveys and Monographs, vol. 132,
\newblock Amer. Math. Soc., 2006
\newblock Mathematical Surveys and Monographs, {\bf vol. 132},
\newblock Amer. Math. Soc., 2006

\bibitem{sch} C. Schlichtkrull, \emph{Units of ring spectra and their traces in algebraic K-theory}, Geometry \& Topology \bf vol 8 no.2 \rm{2004}, 645-673.

\bibitem{waldhausen} F. Waldhausen,  \emph{Algebraic K-theory of spaces},
Algebraic and geometric topology (New Brunswick, N.J., 1983),  Lecture Notes in Math., {\bf1126}, Springer, Berlin, 318Ã419, (1985).
\bibitem{westerland} C. Westerland, \emph{ Equivariant operads, string topology, and Tate cohomology}, Math. Annalen  {\bf vol.  340},   97Ã142   (2008).









     
         }\end{thebibliography}
\end{document}